\documentclass[12pt]{article}
\usepackage{amsthm,amsmath, color}
\usepackage{amsfonts}
\usepackage{graphicx}
\usepackage{latexsym,amsmath}
\usepackage{url}

\newcounter{thm}[section]

\newtheorem{theor}[thm]{Theorem}

\newtheorem{lem}[thm]{Lemma}
\newtheorem{proposition}[thm]{Proposition}
\newcommand{\m}{\hspace*{-.25cm}}
\newcommand{\PP}{\mathbb{P}}
\newcommand{\R}{\mathbb{R}}
\newcommand{\bqa}{\begin{eqnarray*}}
\newcommand{\eqa}{\end{eqnarray*}}

\begin{document}

\title{A General Approach for Cure Models \\ in Survival Analysis}

\author{
{\large Valentin P\textsc{atilea}
\footnote{CREST (Ensai), France. V. Patilea acknowledges support from the research program \emph{New Challenges for New Data}
of LCL and Genes. Email address: \texttt{patilea@ensai.fr}.}}
%This author
%gratefully acknowledges financial support from the Romanian National Authority for Scientific Research, CNCS-UEFISCDI, project
%PN-II-ID-PCE-2011-3-0893.
\and
\addtocounter{footnote}{2}
{\large Ingrid V\textsc{an} K\textsc{eilegom}
\footnote{
ORSTAT, Katholieke Universiteit Leuven, Belgium.  I. Van Keilegom acknowledges support from the European Research Council (2016-2021, Horizon 2020 / ERC grant agreement No.\ 694409), and from IAP Research Network P7/06 of the Belgian State.  Email address: \texttt{ingrid.vankeilegom@kuleuven.be}. }}
}

\date{\today}

\maketitle

\begin{abstract}
In survival analysis it often happens that some subjects under study do not experience the event of interest; they are considered to be `cured'.  The population is thus a mixture of two subpopulations~: the one of cured subjects, and the one of `susceptible' subjects.  When covariates are present, a so-called mixture cure model can be used to model the conditional survival function of the population.  It depends on two components~: the probability of being cured and the conditional survival function of the susceptible subjects.  

In this paper we propose a novel approach to estimate a mixture cure model when the data are subject to random right censoring.  We work with a parametric model for the cure proportion (like e.g.\ a logistic model), while the conditional survival function of the uncured subjects is unspecified.  The approach is based on an inversion which allows to write the survival function as a function of the distribution of the observable random variables.  
This leads to a very general class of models, which allows a flexible and rich modeling of the conditional survival function.  We show the identifiability of the proposed model, as well as the weak consistency and the asymptotic normality of the model parameters.  We also consider in more detail the case where kernel estimators are used for the nonparametric part of the model.  The new estimators are compared with the estimators from a Cox mixture cure model via finite sample simulations.  Finally, we apply the new model and estimation procedure on two medical data sets. 

\end{abstract}

\smallskip



\newpage
\normalsize

\baselineskip 16pt

\pagestyle{plain}
\setcounter{footnote}{0}
\setcounter{equation}{0}

%\end{frontmatter}

%%%%%%%%%%%%%%%%%%%%%
\section{Introduction}
%%%%%%%%%%%%%%%%%%%%%

Driven by emerging applications, over the last two decades there has been an increasing interest for time-to-event analysis models allowing the situation where a fraction of the right censored observed lifetimes corresponds to subjects who will never experience the event. In biostatistics such models including covariates are usually called \emph{cure models} and they allow for a positive \emph{cure fraction} that corresponds to the proportion of patients cured of their disease. For a review %of various types 
of these models in survival analysis, see for instance Maller \& Zhou (2001) or Peng \& Taylor (2014).
%, or Amico \& Van Keilegom (2017).  
Economists sometimes call such models \emph{split population models} (see Schmidt \& Witte 1989), while the reliability engineers refer to them as \emph{limited-failure population life models} (Meeker 1987).  

At first sight, a cure regression model is nothing but a binary outcome, cured versus uncured, regression problem. The difficulty  comes from the fact that the cured subjects are unlabeled observations among the censored data. Then one has to use all the observations, censored and uncensored, to complete the missing information and thus to identify, estimate and make inference on the cure fraction regression function. We propose a general approach
for this task, a tool that provides a general ground for cure regression models. The idea is to start from the laws of the observed variables and to express the quantities of interest, such as the cure rate and the conditional survival of the uncured subjects, as functionals of these laws. 
These general expressions, that we call inversion formulae and we derive with no particular constraint on the %sample 
space of the covariates, are the vehicles that allow for a wide modeling choice, parametric, semiparametric and nonparametric, for both the law of the lifetime of interest and the cure rate. Indeed, the inversion formulae allow to express the likelihood of the binary outcome model as a function of the laws of the observed variables. 
The likelihood estimator of the parameter vector of the cure fraction function is then simply the maximizer of the likelihood obtained by replacing  the laws of the observations by some estimators. With at hand the estimate of the parameter of the cure fraction, the inversion formulae will provide an estimate for the conditional survival of the uncured subjects.  For the sake of clarity, we focus on the so-called mixture cure models with a parametric cure fraction function, the type of model that is most popular among practitioners. Meanwhile, the lifetime of interest is left unspecified.

The paper is organized as follows. In Section \ref{secc_2} we provide a general description of the mixture cure models and next we introduce the  needed notation and  present the inversion formulae on which our approach is built. We finish Section \ref{secc_2} by a discussion of the identification issue and some new  insight on the existing approaches in the literature on cure models.  Section \ref{secc_3} introduces the general maximum likelihood estimator, while in Section \ref{secc_4} we derive the general asymptotic results. A simple bootstrap procedure for making feasible inference is proposed. Section \ref{secc_4} ends with an illustration of the general approach in the case where the conditional law of the observations is estimated by kernel smoothing. In Sections \ref{sec_simul} and \ref{sec_data} we report some empirical results obtained with simulated and two real data sets. Our estimator performs well in simulations and provides similar or more interpretable results in applications compared with a competing logistic/proportional hazards mixture  approach. The technical proofs are relegated to the Appendix.

%%%%%%%%%%%%%%%%%%%%%
\section{The model}\label{secc_2}
%%%%%%%%%%%%%%%%%%%%%

\setcounter{equation}{0}

\subsection{A general class of mixture cure models}\label{sec_2a}

Let $T$ denote (a possible monotone transformation of) the lifetime of interest that takes values in $(-\infty,\infty]$. A cured observation corresponds to the event $\{T=\infty\}$, and in the following this event is allowed to have a positive probability.  Let $X$ be a %$d-$dimensional 
covariate vector with support $\mathcal{X}$ belonging to a general covariate space. The covariate vector could include discrete and continuous components.  The survival function $F_T((t,\infty] \mid x) = \PP(T > t \mid X=x)$ for $t\in \R$ and $x \in \mathcal{X}$ can be written as
\bqa
%S(t \mid x) &=& 1-\phi(x) + \phi(x) S_0(t \mid x), 
F_T((t,\infty] \mid x) &=& 1-\phi(x) + \phi(x) F_{T,0}((t,\infty) \mid x), 
\eqa
where $\phi(x) = \PP(T < \infty \mid X=x)$ and $F_{T,0}((t,\infty) \mid x) = \PP(T > t \mid X=x, T < \infty)$.  Depending on which model is used for $\phi(x)$ and $F_{T,0}(\cdot \mid x)$, one obtains a parametric, semiparametric or nonparametric model, called a  `mixture cure model'.   In the literature, one often assumes that $\phi(x)$ follows a logistic model, i.e.\ $\phi(x) = \exp(a+x^\top b) / [1+\exp(a+x^\top b)]$ for some $(a,b^\top)^\top \in \R^{d+1}$.   Recently, semiparametric models (like a single-index model as in Amico \emph{et al.} 2017) or nonparametric models (as in Xu \& Peng 2014 or L\'opez-Cheda \emph{et al.} 2017) have been proposed.   As for the survival function $F_{T,0}(\cdot \mid x)$ of the susceptible subjects, a variety of models have been proposed, including parametric models (see e.g.\ Boag 1949, Farewell 1982), semiparametric models based on a proportional hazards assumption (see e.g.\ Kuk \& Chen 1992, Sy \& Taylor 2000, Fang \emph{et al.} 2005, Lu 2008; see also Othus \emph{et al.} 2009) or nonparametric models (see e.g.\ Taylor 1995, Xu \& Peng 2014).

In this paper we propose to model $\phi(x)$ parametrically, i.e.\ we assume that $\phi(\cdot)$ belongs to the family of conditional probability functions
$$ \{\phi(\cdot , \beta): \beta \in B\}, $$
where $\phi(\cdot,\beta)$ takes values in the interval $(0,1)$, $\beta$ is the parameter vector of the model and $B$ is the parameter set.  This family could be the logistic family or any other parametric family.   For the survival function $F_{T,0}(\cdot \mid x)$ we do not impose any assumptions in order to have a flexible and rich class of models for $F_T(\cdot \mid x)$ to choose from.   Later on we will see that for the estimation of $F_{T,0}(\cdot \mid x)$ any estimator that satisfies certain minimal conditions can be used, and hence we allow for a large variety of parametric, semiparametric and nonparametric estimation methods.  

As is often the case with time-to-event data, we assume that the lifetime $T$ is subject to random right censoring, i.e.\ instead of observing $T$, we only observe the pair $(Y,\delta)$, where $Y = T\wedge C$,  $\delta = \textbf{1}\{ T\leq C \}$ and $C$ is a non-negative random variable, called the censoring time.  Some identification assumptions are required to be able to identify the conditional law of $T$ from the observed variables $Y$ and $\delta$. Let us assume that
\begin{equation}\label{iden_cd1}
C\perp T \mid X\qquad \text{ and } \qquad  \mathbb{P}(C<\infty)=1.
\end{equation}
The conditional independence between $T$ and $C$ is an usual identification assumption in survival analysis in the presence of covariates. The zero probability at infinity condition for $C$ implies that $\mathbb{P}(C<\infty\mid X)=1$ almost surely (a.s.). This latter mild condition is required if we admit that the observations $Y$ are finite, which is the case in the common applications. 
%, and $\mathbb{P}(T=\infty \mid X=x)>0$ for all $x$.  
For the sake of simplicity, let us also consider the condition
%\footnote{Note that $\mathbb{P}(T=C)=0$ implies $\mathbb{P}(T=C\mid X)=0.$ On the other hand, if  $\mathbb{P}(T=\infty \mid X=x)>0$ for all $x,$ then by the conditional independence of $C$ and $T$ given $X,$  $\mathbb{P}(C<\infty\mid X =x)=1$ for all $x$. }
\begin{equation}\label{iden_cd2}
\mathbb{P}(T=C)=0,
\end{equation}
which is commonly used in survival analysis, and which implies that $\mathbb{P}(T=C\mid X)=0$ a.s.

\subsection{Some notations and preliminaries}

We start with some preliminary arguments, which are valid in general without assuming any model on the functions $\phi$, $F_T$ and $F_C$.  

The observations are characterized by the conditional sub-probabilities
\begin{eqnarray*}
H_1((-\infty,t]\mid x) &=& \mathbb{P}(Y\leq t ,\delta = 1 \mid X=x)\\
H_0((-\infty,t]\mid x) &=& \mathbb{P}(Y\leq t ,\delta = 0 \mid X=x),\qquad t \in \R,\; x\in\mathcal{X} .
\end{eqnarray*}
Then $H((-\infty,t]\mid x) \stackrel{def}{=} \mathbb{P}( Y \leq t \mid X=x ) = H_0((-\infty,t]\mid x) + H_1((-\infty,t]\mid x).$ Since we assume that $Y$ is finite, we have
\begin{equation}\label{finite_H}
H((-\infty,\infty) \mid x) =1,\qquad \forall x\in\mathcal{X}.
\end{equation}
For $j\in\{0,1\}$ and $x\in\mathcal{X},$ let $\tau_{H_j} (x)=\sup\{ t: H_j([t,\infty)\mid x)>0 \}$ denote the right endpoint of the support of the conditional sub-probability $H_j$. Let us define $\tau_{H} (x)$ in a similar way and note that $\tau_{H} (x)=\max\{ \tau_{H_0} (x), \tau_{H_1} (x) \}$.  Note that $\tau_{H_0}(x)$, $\tau_{H_1}(x)$ and $\tau_H(x)$ can equal infinity, even though $Y$ only takes finite values.  For $-\infty < t\leq \infty,$ let us define the conditional probabilities
$$
F_C((-\infty,t]\mid x)= \mathbb{P} (C\leq t\mid X=x)\quad \text{ and } \quad
F_T((-\infty,t]\mid x)= \mathbb{P} (T\leq t\mid X=x),\quad x\in\mathcal{X}.
$$

Let us show how the probability of being cured could be identified from the observations without any reference to a model for this probability. %This generalizes the setup considered by Maller and Zhou (1996) to the case with covariates.
Under conditions (\ref{iden_cd1})-(\ref{iden_cd2}) we can write %the model equations
$$
H_1(dt\mid x) =  F_C ([t,\infty)\mid x)F_T (dt\mid x),
\qquad
H_0(dt\mid x) =  F_T ([t,\infty]\mid x) F_C (dt\mid x),
$$
and
$
H([t,\infty)\mid x) =  F_T ([t,\infty]\mid x) F_C ([t,\infty)\mid x).
$
These equations could be solved and thus they allow to express the functions $F_T (\cdot \mid x)$ and  $F_C (\cdot\mid x)$ in an unique way as explicit transformations of the functions
$H_0(\cdot\mid x)$ and $H_1(\cdot\mid x).$ For this purpose, let us consider the conditional cumulative hazard measures
$$
\Lambda_T (dt \mid x) =
\frac{F_T (dt\mid x)}{F_T ([t,\infty]\mid x)}\quad \text{ and }\quad \Lambda_C (dt \mid x) =
\frac{F_C (dt\mid x)}{F_C ([t,\infty)\mid x)}, \qquad x\in\mathcal{X}.
$$
The model equations yield
\begin{equation}\label{classical_KM}
\Lambda_T (dt \mid x) = \frac{H_1(dt \mid x)}{H([t,\infty)\mid x) }\quad \text{ and }\quad
\Lambda_C (dt \mid x) = \frac{H_0(dt \mid x)}{H([t,\infty)\mid x) }.
\end{equation}
Then, we can write the following functionals of $H_0(\cdot\mid x)$ and $H_1(\cdot\mid x)$ :
\begin{eqnarray}\label{inv_FT_1}
F_T ((t,\infty]\mid x) &= &\prod_{-\infty< s\leq t } \{ 1 - \Lambda_T(ds \mid x)  \},\notag\\
F_C ((t,\infty)\mid x) &=& \prod_{-\infty< s\leq t } \{ 1 - \Lambda_C(ds \mid x)  \}, \qquad t \in \R,
\end{eqnarray}
where $\prod_{s\in A}$ stands for the product-integral over the set $A$ (see Gill and Johansen 1990).

%Let us point out that the support of  $\Lambda_T (dt \mid x)$ and $F_T (\cdot\mid x)$ (resp. $\Lambda_C (dt \mid x)$ and $F_C (\cdot\mid x)$) coincides with the support of $H_1(dt \mid x)$ (resp. $H_0(dt \mid x)$).
Moreover, if $\tau_{H_1}(x)<\infty$, then
\begin{equation*}
\mathbb{P}(T > \tau_{H_1}(x)\mid x ) = \prod_{t\in (-\infty,\tau_{H_1}(x)]} \{ 1 - \Lambda_T(dt \mid x)  \},
\end{equation*}
but there is no way to identify the conditional law of $T$  beyond $\tau_{H_1}(x).$ Therefore, we will impose %that 
\begin{equation} \label{iden_cd3}
\mathbb{P}(T > \tau_{H_1}(x)\mid x )= \mathbb{P}(T =\infty \mid x ),
\end{equation}
i.e.\ $\prod_{t\in \R} \{ 1 - \Lambda_T(dt \mid x)  \} = \prod_{-\infty< t \leq \tau_{H_1}(x)} \{ 1 - \Lambda_T(dt \mid x)\}$.  
%that is
%\begin{equation}\label{id_mass_infty}
%\mathbb{P}(T=\infty\mid x) = \prod_{t\in (0,\infty)} \{ 1 - \Lambda_T(dt \mid x)  \}.
%\end{equation}
Note that if $\tau_{H_1}(x)=\infty,$ condition (\ref{iden_cd3}) is no longer an identification restriction, but just a simple consequence of the definition of $\Lambda_T(\cdot\mid x)$.
Finally, the condition that $\PP(C<\infty)=1$ in (\ref{iden_cd1}) can be re-expressed by saying that we assume that $H_0(\cdot \mid x)$ and $ H_1 (\cdot\mid x)$ are such that
\begin{equation}\label{id_mass_infty_C}
\mathbb{P}(C=\infty\mid x) = \prod_{t\in \R} \{ 1 - \Lambda_C(dt \mid x)  \}=0,\qquad \forall x\in\mathcal{X}.
\end{equation}
Let us point out that this condition is satisfied only if $\tau_{H_1}(x)\leq \tau_{H_0}(x)$.  Indeed, if  $\tau_{H_1}(x)> \tau_{H_0}(x)$ then necessarily $\tau_{H_0}(x)<\tau_H(x)$  and so $H([\tau_{H_0} (x),\infty)\mid x) >0.$ Hence, $\Lambda_C(\R\mid x) = \Lambda_C((-\infty,\tau_{H_0} (x)]\mid x)<\infty$, and thus $\mathbb{P}(C=\infty\mid x)>0$, which contradicts %is in contradiction with 
(\ref{id_mass_infty_C}).  

It is important to understand that \emph{any} two conditional sub-probabilities $H_0(\cdot \mid x)$ and $H_1 (\cdot\mid x)$ satisfying conditions (\ref{iden_cd1}) and (\ref{iden_cd2}) %and (\ref{iden_cd3}) %(\ref{id_mass_infty_C}) 
define uniquely $F_T (\cdot \mid x) $ and  $F_C (\cdot \mid x) $. Indeed, $F_T (\cdot \mid x) $ is precisely the probability distribution of $T$ given $X=x$ with all the mass beyond $\tau_{H_1}(x)$ concentrated at infinity. In general, $F_T (\cdot \mid x) $ and $F_C (\cdot \mid x) $ are only functionals of $H_0 (\cdot \mid x) $ and $H_1 (\cdot \mid x) $.

We will assume conditions (\ref{iden_cd1}), (\ref{iden_cd2})  and (\ref{iden_cd3}) throughout the paper.

%\subsection{The two-component mixture cure model}
\subsection{A key point for the new approach: the inversion formulae}

Write
\begin{eqnarray*}%\label{eq_H}
H([t,\infty)\mid x ) &=& F_T ([t,\infty]\mid x) F_C ([t,\infty)\mid x) \notag\\
&=& F_T ([t,\infty)\mid x) F_C ([t,\infty)\mid x)
+ \mathbb{P}( T=\infty \mid x)F_C ([t,\infty)\mid x),
\end{eqnarray*}
and thus
\begin{equation}\label{eq_Hbis}
F_T ([t,\infty)\mid x) = \frac{H([t,\infty)\mid x ) - \mathbb{P}( T=\infty \mid x)F_C ([t,\infty)\mid x)}{F_C ([t,\infty)\mid x)}.
\end{equation}
Consider the conditional cumulative hazard measure for the finite values of the lifetime of interest:
$$
\Lambda_{T,0} (dt \mid x) \stackrel{def}{=} \frac{F_{T,0} (dt\mid x)}{F_{T,0} ([t,\infty)\mid x)}= \frac{F_T (dt\mid x)}{F_T ([t,\infty)\mid x)}
$$
for $t \in \R$.  Since $H_1(dt\mid x) =  F_C ([t,\infty)\mid x)F_T (dt\mid x)$, using relationship  (\ref{eq_Hbis}) we obtain
\begin{eqnarray}\label{inversion1}
\Lambda_{T,0} (dt \mid x) &=& \frac{H_1(dt \mid x)}{H([t,\infty)\mid x) - \mathbb{P}(T=\infty \mid x)F_C ([t,\infty)\mid x)}.
%\\
%&=& \frac{H_1(dt \mid x)}{F_T([t,\infty)\mid x)F_C ([t,\infty)\mid x)}.\notag
\end{eqnarray}
%Let us note the dependence of the conditional cumulative hazard function on the parameter $\beta$.
Next, using the product-integral we can write
%\footnote{Note that it is not necessarily true that
%\begin{equation*}%\label{inv_FT}
%F_T ((t,\infty]\mid x) = \prod_{0<s\leq t } \{ 1 - \Lambda_T(ds \mid x, T < \infty)  \}
%\end{equation*}
%because $\Lambda_T(ds \mid x, T < \infty )$ is not defined as
%$$
%\Lambda_T (dt \mid x, T < \infty) = \frac{F_T (dt\mid x)}{F_T ([t,\infty]\mid x)}
%$$
%}
\begin{equation}\label{inv_FT}
F_{T,0} ((t,\infty)\mid x) = \prod_{-\infty < s \leq t } \{ 1 - \Lambda_{T,0}(ds \mid x)  \},\quad t \in \R, \; x\in\mathcal{X}.
\end{equation}
%Note that
%$$
%F_T^{(1)}([0,\infty)\mid x ) = \lim_{t\downarrow 0} \prod_{s\in(0,t] } \{ 1 - \Lambda_T(ds \mid x, T<\infty)  \} =1
%$$
%and
%\begin{equation}\label{ft_ft}
%F_T([t,\infty)\mid x ) = F_{T,0}([t,\infty)\mid x ) \mathbb{P}( T< \infty \mid x),\qquad \forall \; 0 \leq t<\infty.
%\end{equation}
%This was expected since when the model is correct $F_T^{(1)}(\cdot\mid x )$ is precisely the law of $T$ given that $X=x$ and  $T$ is finite.

Let us recall that $F_C (\cdot\mid x)$ can be written as a transformation of $H_0(\cdot\mid x)$ and $H_1(\cdot\mid x)$, see equations (\ref{classical_KM}) and (\ref{inv_FT_1}). This representation is not surprising
%\begin{equation}\label{inv_FC}
%F_C ((t,\infty)\mid X) = \prod_{s\leq t } \{ 1 - \Lambda_C(ds \mid X)  \} =
%\prod_{s\leq t } \left\{ 1 - \frac{H_0 (dt\mid X)}{H ([t,\infty)\mid X)}\right\}.
%\end{equation}
since  we can consider $C$ as a lifetime of interest and hence $T$ plays the role of a censoring variable. Hence, estimating the conditional distribution function  $F_C (\cdot\mid x)$ should not be more complicated than in a classical conditional Kaplan-Meier setup, since the fact that $T$ could be equal to infinity with positive conditional probability is irrelevant when estimating $F_C (\cdot\mid x)$.

Finally, the representation of $F_C (\cdot\mid x)$ given in equation (\ref{inv_FT_1}), plugged into equation (\ref{inversion1}),  allows to express $\Lambda_{T,0}(\cdot\mid x)$, and thus $F_{T,0}(\cdot\mid x)$, as maps of $\mathbb{P}(T=\infty\mid x)$ and the measures $H_0(\cdot\mid x)$ and $H_1(\cdot\mid x).$ This will be the key element for providing more insight in the existing approaches and the starting point of our new approach.

\subsection{Model identification issues}

Let us now investigate the identification issue.  Recall that our model involves the functions $F_T(\cdot\mid x)$, $F_C (\cdot\mid x)$ and $\phi(\cdot,\beta)$, and the assumptions (\ref{iden_cd1}), (\ref{iden_cd2}) and (\ref{iden_cd3}). %Moreover, the model assumes that the functions $H_0$ and $H_1$ are given by 
%\begin{equation}\label{cure_mixture}
%H_1([0,t]\mid x) = \mathbb{P}(Y\leq t ,\delta = 1 \mid x) =  \mathbb{P}(Y\leq t ,\delta = 1 , T<\infty \mid x), \qquad 0\leq t < \infty, \;x\in\mathcal{X},
%\end{equation}
%and
%\begin{eqnarray}\label{eq_H0}
%H_0(dt\mid x) &=&  F_T ([t,\infty]\mid x) F_C (dt\mid x)\notag\\
%&=& [ F_T ([t,\infty)\mid x) + \mathbb{P}(T=\infty\mid x) ] F_C (dt\mid x).
%\end{eqnarray}
For a fixed value of the parameter $\beta,$ and for $t \in \R$ and $x\in\mathcal{X}$, let
\begin{equation}\label{inv_LamT_beta}
\Lambda_{T,0}^\beta(dt\mid x) = \frac{H_1(dt \mid x)}{H([t,\infty)\mid x) - [1-\phi(x,\beta)]F_C ([t,\infty)\mid x)},
\end{equation}
and
\begin{equation}\label{inv_FT_beta}
F_{T,0}^\beta ((t,\infty)\mid x) = \prod_{-\infty < s \leq t } \left\{ 1 -
\Lambda_{T,0}^\beta (ds\mid x)\right\}.
\end{equation}

Let $F_{Y,\delta} (\cdot,\cdot\mid x)$ denote the conditional law of $(Y,\delta)$ given $X=x$. Moreover, let
$$
F_{Y,\delta}^\beta (dt,1\mid x) = \phi(x,\beta) F_C((t,\infty)\mid x) F_{T,0}^\beta (dt\mid x)
$$
and
$$
F_{Y,\delta}^\beta (dt,0\mid x) = [ F_{T,0}^\beta ((t,\infty)\mid x)\phi(x,\beta) + 1-\phi(x,\beta)] F_C(dt\mid x).
$$
These equations define a conditional law for the observations $(Y,\delta)$ based on the model. More precisely, for a choice of $F_T(\cdot\mid x)$,  $F_C(\cdot\mid x)$ and $\beta$, the model yields a conditional law for $(Y,\delta)$ given $X=x$.
If the model is correctly specified, there exists a value $\beta_0$ such that
\begin{equation}\label{def_beta_0}
F_{Y,\delta} (\cdot,\cdot\mid x) = F_{Y,\delta}^{\beta_0} (\cdot,\cdot\mid x), \quad \forall x\in\mathcal{X}.
\end{equation}
%for all $x\in\mathcal{X}$.  %The value $\beta_0$ could be alternatively defined as a value of the parameter such that
%\begin{equation}\label{beta_H}
%\phi(X,\beta_0) = 1- \prod_{t\in (0,\infty)} \left\{ 1 - \frac{H_1(dt \mid X)}{H([t,\infty)\mid X)}  \right\},\qquad \text{almost surely}.
%\end{equation}
%\color{red} INGRID, please read the footnote \color{black}\footnote{Let us notice that one could identify the parameter $\beta_0$ using this characterization. One could also build an estimate of $\beta_0$ from this characterization. Indeed, let $\widehat H_0$ and $\widehat H_1$ be some estimates. For each observed $X_i$, build an estimate of $\widehat \phi (X_i) = 1 - \prod_{t\in (0,\infty)} \left\{ 1 - \frac{\widehat H_1 (dt \mid X_i)}{\widehat H([t,\infty)\mid X_i)}  \right\}$. Estimate $\beta$ by least squares between $\phi (X_i , \beta) $ and $\widehat \phi (X_i )$. This idea could be used in applications to determine a good starting point in our algorithm for calculating the estimator.  BIG QUESTION: would be this $\sqrt{n}-$consistent? If yes (probably), why do we do like us, more complicated, and not do in this way ?????? }

The remaining question is whether the true value of the parameter is identifiable. In other words, one should check if, given the conditional subdistributions $H_0 (\cdot\mid x)$ and $H_1 (\cdot\mid x),$ $x\in\mathcal{X},$ there exists a unique $\beta_0$ satisfying condition (\ref{def_beta_0}).  For this purpose we impose the following mild condition:
\begin{equation}\label{model_inden}
\phi(X,\beta)%\mathbf{1}\{X\in\mathcal{X}\}
=\phi(X,\widetilde \beta)
%\mathbf{1}\{X\in\mathcal{X}\}
,\;\; \text{almost surely} \quad \Rightarrow \beta=\widetilde \beta,
\end{equation}
and we show that
$$
F_{Y,\delta}^{\beta_0} (\cdot,\cdot\mid x) = F_{Y,\delta}^{\widetilde\beta} (\cdot,\cdot\mid x) ,\;\; \forall x\in\mathcal{X} \;\; \Rightarrow \;\; \phi(x,\beta_0) =\phi(x,\widetilde \beta), \;\; \forall x\in\mathcal{X}.
$$
Indeed, if $F_{Y,\delta}^{\beta_0} (\cdot,\cdot\mid x) = F_{Y,\delta}^{\widetilde\beta} (\cdot,\cdot\mid x),$ then for any $x$,
$$
\phi(x,\beta_0) F_C((t,\infty)\mid x) F_{T,0}^{\beta_0} (dt\mid x) = \phi(x,\widetilde \beta) F_C((t,\infty)\mid x) F_{T,0}^{\widetilde \beta} (dt \mid x),
$$
for all $t\in(-\infty,\tau_H(x)]\cap \R.$
Our condition (\ref{id_mass_infty_C}) guarantees that $\tau_H(x)=\tau_{H_0}(x)$ so that $F_C((t,\infty)\mid x)$ should be necessarily positive for $t\in(-\infty,\tau_H(x)),$ and thus could be simplified in the last display. Deduce that
$$
\phi(x,\beta_0) F_{T,0}^{\beta_0} (dt\mid x) = \phi(x,\widetilde \beta) F_{T,0}^{\widetilde \beta} (dt\mid x),
$$
for all $t\in(-\infty,\tau_H(x))$.  Finally, recall that by construction, in the model we consider, $ F_{T,0}^\beta ((-\infty,\infty)\mid x)=1$ for any $\beta$ such that $F_{Y,\delta}^{\beta} (\cdot,\cdot\mid x)$ coincides with the conditional law of $(Y,\delta)$ given $X=x$, for any $x$. 
Thus taking integrals on $(-\infty,\infty)$ on both sides of the last display we obtain $\phi(\cdot,\beta_0) =\phi(\cdot,\widetilde \beta)$. Let us gather these facts in the following statement.

\begin{theor}\label{th_ident}
%Under the conditions (\ref{id_mass_infty_C})
Under conditions (\ref{iden_cd1}), (\ref{iden_cd2}), (\ref{iden_cd3}) and (\ref{model_inden}) the model is identifiable.
\end{theor}

%%%%%%%%%%%%
%At this stage we could consider $\overline{\mathcal{X}} = \mathcal{X}$ but for estimation %purposes, in order to keep denominators away from zero, it may be preferable to take a %smaller $\overline{\mathcal{X}}$.
%%%%%%%%%%%

%\bigskip
%\noindent
%{\bf Remark 2.1}: 

\subsection{Interpreting the previous modeling approaches}\label{sec_inter}

We suppose here that the function $\phi(x)$ follows a logistic model, and comment on several models for $F_{T,0}$ that have been considered in the literature.

%\begin{enumerate}
%\item 

\subsubsection{Parametric and proportional hazards mixture model} 

In a parametric modeling, one usually supposes that $\tau_{H_0}(x) = \tau_{H_1}(x) = \infty$ and that $\Lambda_{T,0} (\cdot \mid x)$ belongs to a parametric family of cumulative hazard functions, like for instance the Weibull model; see Farewell (1982).

Several contributions proposed a more flexible semiparametric proportional hazards (PH) approach; see Fang \emph{et al.} (2005), Lu (2008) and the references therein. In such a model one imposes a PH structure for the $\Lambda_{T,0}(\cdot\mid x)$ measure. More precisely, it is supposed that
\begin{equation*}\label{cox}
\Lambda_{T,0} (dt \mid x) = \frac{H_1(dt \mid x)}{H([t,\infty)\mid x) - [1-\phi(x,\beta)]F_C ([t,\infty)\mid x)}=\exp(x^\top \gamma)\Lambda_0(dt),
\end{equation*}
where $\gamma$ is some parameter to be estimated and $\Lambda_0(\cdot)$ is an unknown baseline cumulative hazard function.
Our inversion formulae reveal that in this approach the parameters $\gamma$ and $\Lambda_0$ depend on the observed conditional measures $H_0(\cdot\mid x)$ and $H_1(\cdot\mid x),$ but also on the parameter $\beta$.  The same is true for the parametric models. 

%\item 

\subsubsection{Kaplan-Meier mixture cure model}

Taylor (1995) suggested to estimate $F_{T,0}$ using a Kaplan-Meier type estimator.
% and to suppose that $H(\cdot\mid x)$ does not depend on $x$.  
With such an approach one implicitly assumes that the law of $T$ given $X$ and given that $T<\infty$ does not depend on $X$. This is equivalent to supposing that $\Lambda_{T,0} (\cdot \mid x)=\Lambda_{T,0} (\cdot)$.  Next, to estimate $\Lambda_{T,0} (\cdot)$ one has to modify the unconditional version of the usual inversion formulae (\ref{classical_KM}) to take into account the conditional probability of the event $\{T=\infty\}$. Following Taylor's approach we rewrite (\ref{inversion1}) as
$$
\Lambda_{T,0} (dt) =\frac{H_1(dt \mid x)}{H_1([t,\infty)\!\mid\! x) \!+ \int_{[t,\infty)} \!\left\{1- \frac{1\!-\phi(x,\beta)} {\phi(x,\beta)F_{T,0}([s,\infty)) +1-\phi(x,\beta)} \right\}\!H_0 (ds\mid x)}.
$$
Next, assume that the last equality remains true if $H_0(dt \mid x)$ and $H_1(dt \mid x)$ are replaced by their unconditional versions, that is assume that
\begin{equation}\label{Taylor_imp}
\Lambda_{T,0} (dt) = \frac{H_1(dt)}{H_1([t,\infty)) + \int_{[t,\infty)} \left\{1- \frac{1-\phi(x,\beta)} {\phi(x,\beta)F_{T,0}([s,\infty)) +1-\phi(x,\beta)} \right\}H_0 (ds)}.
\end{equation}
See equations (2) and (3) in Taylor (1995). The equation above could be solved iteratively by a EM-type procedure: for a given $\beta$ and an iteration
$F_{T,0}^{(m)}(\cdot),$ build $\Lambda_{T,0}^{(m+1)} (dt)$ and the updated estimate $F_{T,0}^{(m+1)}(\cdot)$; see Taylor (1995) for the details.  
Let us point out that even if $(T,C)$ is independent of $X$ and thus $H_1(\cdot\mid x)$ does not depend on $x$, the subdistribution $H_0(\cdot \mid x)$ still depends on $x$, since
\begin{eqnarray*}%\label{eq_H0}
H_0(dt\mid x) &=&  F_T ((t,\infty]\mid x) F_C (dt\mid x)\notag\\
&=& [ F_T ((t,\infty)\mid x) + \mathbb{P}(T=\infty\mid x) ] F_C (dt\mid x).
\end{eqnarray*}
%see equation (\ref{eq_H0}). 
Hence, a more natural form of equation (\ref{Taylor_imp}) is
\begin{equation*}%\label{Taylor_imp_corr}
\Lambda_{T,0} (dt) = \frac{H_1(dt)}{H_1([t,\infty)) + \int_{[t,\infty)} \left\{1- \frac{1-\phi(x,\beta)} {\phi(x,\beta)F_{T,0}([s,\infty)) +(1-\phi(x,\beta))} \right\}H_0 (ds\mid x)}.
\end{equation*}
The investigation of a EM-type procedure based on the latter equation will be considered elsewhere.

%\end{enumerate}

%\qquad
%
%{\color{black} It remains to check the identification issue and the role of $\tau$. Probably it has no role for the identification. Check also the paper of Li, Taylor, Sy (2001).  }

%\subsection{Promotion time cure model}
%
%An alternative approach to mixture models for the
%cure rate is the so-called \emph{promotion
%time cure model}. See, for instance, Zheng \emph{et al.} (2006) and Tsodikov \emph{et al.} (2003). In such an approach it is assumed that
%$$
%F_T((t,\infty]\mid x)=\exp\{-\theta(x)F^*([0,t])\}
%$$
%where $\theta(\cdot)$ is a know link function depending on some parameter and $F^*(\cdot)$ is a baseline distribution function independent of $x$. In particular, $\mathbb{P}(T=\infty\mid X=x)= \exp\{-\theta(x)\}.$
%
%In our approach, this corresponds to the following representation of the conditional cumulative hazard  defined by (\ref{inversion1}),
%\begin{eqnarray*}\label{ptm}
%\Lambda_T ([0,t] \mid x, B=1) & = & \int_{[0,t]} \frac{H_1(ds \mid x)}{H([s,\infty)\mid x) - \exp\{-\theta(x)\} F_C ([s,\infty)\mid x)} ds\\
%& = & -\log\left[ \frac{\exp\{-\theta(x)F^*([0,t])\} - \exp\{-\theta(x)\}}{1-\exp\{-\theta(x)\}} \right].
%\end{eqnarray*}
%
%{\color{black} Any additional comment ? }
%

%%%%%%%%%%%%%%%%%%%%
\section{Maximum likelihood estimation}\label{secc_3}
%%%%%%%%%%%%%%%%%%%%

\setcounter{equation}{0}

Let $(Y_i,\delta_i,X_i)$ ($i=1,\ldots,n$) be a sample of $n$ i.i.d.\ copies of the vector $(Y,\delta,X)$.

We use a likelihood approach based on formulae (\ref{inversion1}) and (\ref{inv_FT_1}) to build an estimator of $\phi(\cdot,\beta)$ and $F_{T,0}^\beta(\cdot\mid x).$ To build the likelihood we use estimates $\widehat H_k (\cdot \mid x)$ of the subdistributions
$H_k (\cdot \mid x),$ $k\in\{0,1\}.$ These estimates are constructed with the sample of $(Y,\delta,X),$ without reference to any model for the conditional probability $\mathbb{P}(T<\infty\mid x).$ At this stage it is not necessary to impose a particular form for $\widehat H_k (\cdot \mid x)$. To derive the asymptotic results we will only impose that these estimators satisfy some mild  conditions.  Let $\widehat F_{T,0}^\beta(\cdot\mid x)$ be defined as in equations (\ref{inv_LamT_beta}) and (\ref{inv_FT_beta}) with $\widehat H_0 (\cdot \mid x)$ and $\widehat H_1 (\cdot \mid x)$ instead of $H_0 (\cdot \mid x)$ and $H_1 (\cdot \mid x)$, that is
$$
\widehat F_{T,0}^\beta ((t,\infty)\mid x) = \prod_{-\infty < s \leq t } \left\{ 1 -
\frac{\widehat H_1(ds \mid x)}{\widehat H([s,\infty)\mid x) - [1-\phi(x,\beta)]\widehat F_C ([s,\infty)\mid x)}\right\},
$$
where $\widehat F_{C}(\cdot\mid x)$ is the estimator obtained from equations (\ref{classical_KM}) and (\ref{inv_FT_1}) but with $\widehat H_0 (\cdot \mid x)$ and $\widehat H_1 (\cdot \mid x)$, i.e.
$$
\widehat F_C ((t,\infty)\mid x) = \prod_{-\infty< s\leq t } \left\{ 1 - \frac{\widehat H_0(ds \mid x)}{\widehat H([s,\infty)\mid x) }  \right\}.
$$

%\subsection{Likelihood}

Let $f_X$ denote the density of the covariate vector with respect to some dominating measure. The contribution of the observation $(Y_i,\delta_i, X_i)$ to the likelihood when $\delta_i= 1 $
is then
$$
\phi(X_i,\beta) \widehat F_{T,0}^\beta(\{Y_i\}\mid X_i) f_X(X_i) \widehat F_C ((Y_i,\infty)\mid X_i),
$$
while the contribution when $\delta_i= 0$ is
$$
\widehat F_C (\{Y_i\}\mid X_i)f_X(X_i) [\phi(X_i,\beta)\widehat F_{T,0}^\beta ((Y_i,\infty)\mid X_i) + 1 - \phi(X_i,\beta)].
$$
Since the laws of the censoring variable and of the covariate vector do not carry information of the parameter $\beta$, 
%Under the assumption of cure non-informative censoring (see Assumption qqq1) and given that the law of the covariates does not depend on $\beta$ (see Assumption qqq2), 
we can drop the factors $\widehat F_C ((Y_i,\infty)\mid X_i) f_X(X_i)$ and $\widehat F_C (\{Y_i\}\mid X_i)f_X(X_i).$ 
%Moreover,  by construction $\widehat F_{T,0}^\beta(\{Y_i\}\mid X_i)=0$ when $\delta_i=0$.
Hence the criterion to be maximized with respect to $\beta\in B$ is $\widehat L_n(\beta)$ where
$$
\widehat L_n(\beta)=\prod_{i=1}^n \Big\{ \phi(X_i, \beta)\widehat F_{T,0}^\beta(\{Y_i\}\mid X_i) \Big\}^{\delta_i} \Big\{ \phi(X_i,\beta)\widehat F_{T,0}^\beta ((Y_i,\infty)\mid X_i) + 1 - \phi(X_i,\beta) \Big\}^{1-\delta_i}.
$$
The estimator we propose is
\begin{equation}\label{estim_def}
\widehat \beta = \arg\max_{\beta \in B} \log \widehat L_n(\beta).
\end{equation}

Let us review the identification issue in the context of the likelihood estimation approach. If  conditions (\ref{iden_cd1}), (\ref{iden_cd2}) and (\ref{iden_cd3}) hold true and the parametric model for the conditional probability of the event $\{T=\infty\}$ is correct and identifiable in the sense of condition (\ref{model_inden}), the true parameter $\beta_0$ is the value identified by condition (\ref{def_beta_0}).  This is the conclusion of Theorem \ref{th_ident} above. It remains to check that the proposed likelihood approach allows to consistently estimate $\beta_0.$

%We can also write
%$$
%L_n(\beta)= \prod_{i=1}^n \{\phi(X_i; \beta) F_{T,\beta}^{(1)}(\{Y_i\}\mid X_i) \} ^{\delta_i} \{1-  \phi(X_i; \beta)F_{T,\beta}^{(1)} ([0,Y_i)\mid X_i)  \}^{1-\delta_i}
%$$
%or, given that $F_{T,\beta}^{(1)}(\{Y_i\}\mid X_i) = \Lambda_{T,\beta}^{(1)} (\{Y_i\} \mid X_i) F_{T,\beta}^{(1)}([Y_i,\infty)\mid X_i) ,$
%\begin{eqnarray*}
%L_n(\beta)&=&\prod_{i=1}^n \{\Lambda_{T,\beta}^{(1)} (\{Y_i\} \mid X_i )\}^{\delta_i} \{ \phi(X_i; \beta)F_{T,\beta}^{(1)}([Y_i,\infty)\mid X_i) \} ^{\delta_i} \\
%&&\qquad \times \{ \phi(X_i;\beta)F_{T,\beta}^{(1)} ([Y_i,\infty)\mid X_i) + 1 - \phi(X_i;\beta) \}^{1-\delta_i}.
%\end{eqnarray*}
%

%{\color{black} Iterative methods could be useful!!}

%\subsection{Method identification issues}\label{sec_like}
%{\color{black}

Let
\begin{equation}\label{bernoulli_lik}
\log p(t,\delta,x;\beta) = \delta \log p_1(t,x;\beta) + (1-\delta)  \log p_0(t,x;\beta),
\end{equation}
with
\begin{equation}\label{p_bernoulli1}
p_1(t,x;\beta) = \frac{\phi(x,\beta) F_{T,0}^\beta(dt\mid x) F_C ((t,\infty)\mid x)}{H_1(dt\mid x)},
\end{equation}
\begin{equation}\label{p_bernoulli2}
p_0(t,x;\beta) =\frac{ F_C (dt\mid x)[\phi(x,\beta)F_{T,0}^\beta ((t,\infty)\mid x) + 1 - \phi(x,\beta)]}{H_0(dt\mid x)}.
\end{equation}
Following a common notational convention, see for instance Gill (1994), here we treat $dt$ not just as the length of a small time interval $[t, t + dt)$ but also as the name of the interval
itself. Moreover, we use the convention $0/0=1.$  Let us notice that, up to additive terms  not containing $\beta,$ the function $\beta\mapsto \mathbb{E}[\log p(Y,\delta,X;\beta)]$ is expected to be the limit of the random function $\log \widehat L_n(\cdot).$ Hence, a minimal condition for guaranteeing the consistency of the likelihood estimation approach is that $\beta_0$ is the maximizer of the limit likelihood criterion $\mathbb{E}[\log p(Y,\delta,X;\cdot)]$. This is proved in the following proposition using a Bernoulli sample likelihood inequality.  The proof is given in the Appendix.

\begin{proposition}\label{ident_lik}
Suppose that conditions (\ref{iden_cd1}), (\ref{iden_cd2}), (\ref{iden_cd3}) and  (\ref{model_inden}) hold true. If $\beta_0$ is the value of the parameter defined by equation (\ref{def_beta_0}), then for any $\beta\neq\beta_0$,
$$
\mathbb{E}\left[
%\textbf{1}\{X\in\mathcal{X}\}\int_{[0,\infty)}
\log p(Y,\delta,X;\beta) \right] < \mathbb{E}\left[
%\textbf{1}\{X\in\mathcal{X}\}\int_{[0,\infty)}
\log p(Y,\delta,X;\beta_0) \right].
$$
%\color{black} some integrability conditions might be necessary here to avoid $\infty - \infty$ \color{black}
\end{proposition}

%%%%%%%%%%%%%%%%%%%%%
\section{General asymptotic results}\label{secc_4}
%%%%%%%%%%%%%%%%%%%%%

\setcounter{equation}{0}

Little assumptions were needed for our analysis so far.
To proceed further with the asymptotic results we need to be more specific with respect to several aspects. In order to prove consistency, we have to control the asymptotic behavior of $\widehat L_n (\beta)$ along sequences of values of the parameter $\beta.$ Such a control requires a control of denominators like
$$
\widehat H([t,\infty)\mid x) - (1-\phi(x,\beta)) \widehat F_C ([t,\infty)\mid x)
$$
on the support of $H_1(\cdot\mid x),$ uniformly with respect to $x.$
 %For such a control, in addition to assumptions on the model $\{\phi(\cdot ; \beta): \beta \in B\}$ and the covariates, one should check that $\widehat H_0([t,\infty)\mid x)$ and $\widehat H_1([t,\infty)\mid x)$ stay away from zero in the right tail, uniformly with respect to $x.$
A usual way to deal with this technical difficulty is to consider a finite threshold $\tau$ beyond which no uncensored lifetime is observed, i.e.
\begin{equation}\label{hhhh_0}
 \inf_x H_1((-\infty,\tau]\mid x) =1.
\end{equation}
Moreover, to be able to keep denominators away from zero, we require the condition
\begin{equation}\label{hhhh}
\inf_{x \in\mathcal{X} } H_1(\{\tau\}\mid x) H_0((\tau,\infty)\mid x) >0.
\end{equation}
In particular, this condition implies
$$
\tau_{H_0}(x) > \tau_{H_1}(x) =\tau,\qquad x\in\mathcal{X}.
$$
Moreover, given condition (\ref{iden_cd2}), necessarily $H_0(\{\tau\}\mid x)=0$, $\forall x.$ This means that $F_C(\{\tau\}\mid x)=0$, $\forall x.$
This constraint on $H_0(\{\tau\}\mid x)$ could be relaxed at the expense of suitable adjustments of the inversion formulae. For simplicity, we keep condition (\ref{iden_cd2}). Let us also notice that condition (\ref{hhhh}) implies %that
$
\inf_x F_{C} ((\tau,\infty)\mid x) >0,
$
and %, for any $\beta,$
$
\inf_x F_{T,0}^\beta (\{\tau\}\mid x) >0
$, $\forall \beta.$
%Let us point out that without any information on $T$ beyond the threshold,  the conditional distribution of $T$ could not be identified beyond $\tau,$ and thus, by definition $
%\mathbb{P}(T>\tau\mid x) = \mathbb{P}(T=\infty\mid x).$

Conditions like in equations (\ref{hhhh_0})-(\ref{hhhh}) are more or less explicitly used in the literature of cure models. Sometimes $\tau$ is justified  as representing a total follow-up of the study. For instance, Lu (2008) supposes that $Y=\min\{ T,\min (C,\tau)\}$ and $\delta = \mathbf{1}\{ T\leq \min(C,\tau) \}, $ where $T= \eta T^* + (1-\eta) \infty,$ with $T^*<\infty$ and $\eta\in\{0,1\}.$  The conditional probability of being cured is precisely the conditional probability of the event $\{\eta = 0\}.$ Next, Lu (2008) supposes that
$
\inf _x \mathbb{P} (\tau \leq T \leq  C  \mid x ) >0,
$
and  $\Lambda_0(\tau)<\infty,$ where $\Lambda_0(\cdot)$ is the cumulative hazard function of $T^*.$ All these conditions together clearly imply our conditions  (\ref{hhhh_0})-(\ref{hhhh}).

Fang \emph{et al.} (2005) implicitly restrict the uncensored lifetimes to some compact interval $[0,\tau]$ and suppose
$
\mathbb{E}(\delta \mathbf{1}\{Y\geq \tau\}) >0.
$
This could be possible only if $H_{1}(\{\tau\}\mid x)>0$ for a set of values $x$ with positive probability. In a proportional hazards context with the covariates taking values in a bounded set, as assumed by Fang \emph{et al.} (2005), this is equivalent to $ H_{1}(\{\tau\}\mid x)\geq c >0$ for almost all $x,$ for some constant $c.$
%On the other hand, they impose $\int_{[0,\infty)} F_C((t,\infty)) \Lambda_0(dt) < \infty, $ where $\Lambda_0$ is the baseline cumulative hazard function. Then, necessarily  $\int_{[0,\infty)} F_C((t,\infty)\mid x) \Lambda_0(dt) < \infty < \infty,$ for almost all $x.$ In a proportional hazard context this is equivalent to $\int_{[0,\infty)} F_C((t,\infty)\mid x) \Lambda_T(dt\mid x, T< \infty) < \infty,$ for almost all $x.$
%This implies $H_{0}((\tau,\infty)\mid x),$  for almost all $x.$

The fact that technical conditions similar to our conditions (\ref{hhhh_0})-(\ref{hhhh}) could be traced in the  cure models literature is not unexpected in view of our Section \ref{sec_inter}. Indeed, the existing approaches could be interpreted through our inversion formulae and thus the technical problems we face %when proceeding to 
in the asymptotic investigation are expected to be also present in the alternative approaches.

\subsection{Consistency} \label{consistency}

Let us sketch the arguments we use in the proof of Theorem \ref{consist_prop} below for deriving the consistency of $\widehat \beta$.
On one hand, if the conditional subdistributions $H_k(\cdot\mid x)$ are given, one can build the purely parametric likelihood
\begin{equation}\label{lcal_n}
\mathcal{L}_n(\beta)=\prod_{i=1}^n \!\Big\{ \phi(X_i, \beta) F_{T,0}^\beta(\{Y_i\}\mid X_i) \Big\} ^{\delta_i} \!\Big\{ \phi(X_i,\beta) F_{T,0}^\beta ((Y_i,\infty)\mid X_i) + 1 - \phi(X_i,\beta) \Big\}^{1-\delta_i}\!\!.
\end{equation}
By construction, $\mathcal{L}_n(\beta)$  is a functional of $H_0(\cdot\mid x)$ and $H_1(\cdot \mid x),$ $x\in\mathcal{X},$ while $\widehat L_n(\beta)$ is a functional of the estimated versions of $H_0(\cdot\mid x)$ and $H_1(\cdot \mid x).$ Hence, a prerequisite condition for deriving the consistency of our semiparametric estimator $\widehat \beta$ is the consistency of the infeasible maximum likelihood estimator
$$
\widetilde \beta = \arg\max_{\beta\in B} \log \mathcal{L}_n(\beta).
$$
%Next, up to some technical modifications that simplify the arguments, the basic idea is to show that
A necessary condition for the consistency of $\widehat \beta$ is 
\begin{equation}\label{consis1}
\sup_{\beta\in B}|\log \widehat {L}_n(\beta) - \log \mathcal{L}_n(\beta)| = o_{\mathbb{P}}(1).
\end{equation}
We then have that
$$
\log \mathcal{L}_n(\widehat \beta) \geq \log \mathcal{L}_n(\widetilde \beta) - o_{\mathbb{P}}(1).
$$
From this we will derive the consistency of $\widehat \beta$ using Section 5.2 in van der Vaart (1998).  See the proof in the Appendix for details. 
%and this imply the consistency of $\widehat\beta.$ See, for instance, section 5.2 in van der Vaart (1998).

To prove condition (\ref{consis1}), we have to guarantee the uniform convergence of $\widehat H_k-H_k$ as stated in Assumption (AC1) below.  
%that
%\begin{equation*}%\label{unif1}
 %\sup_{x\in \mathcal{X}} \sup_{t\in (-\infty,\tau]} | \widehat H_k([t,\infty)\mid x) - H_k([t,\infty)\mid x) | = o_{\mathbb{P}}(1),\qquad k\in\{0,1\}.
%\end{equation*}
%with some $\overline{\mathcal{X}}\subset\mathcal{X}$ satisfying the condition (\ref{model_inden}).
%The uniform convergence of $\widehat H_k-H_k$ will imply that
Indeed, this uniform convergence will imply 
\begin{equation}\label{unif2}
\sup_{x\in \mathcal{X}} \sup_{t\in (-\infty,\tau]} | \widehat F_C([t,\infty)\mid x) - F_C([t,\infty)\mid x) | = o_{\mathbb{P}}(1),
\end{equation}
and %that 
\begin{equation}\label{unif3}
\sup_{\beta\in B} \; \sup_{x\in \mathcal{X}}\; \sup_{t\in (-\infty,\tau]} | \widehat  F_{T,0}^\beta([t,\infty)\mid x) - F_{T,0}^\beta([t,\infty)\mid x) | = o_{\mathbb{P}}(1).
\end{equation}
See Lemma \ref{F_C_et_al} in the Appendix. The uniform 
convergence in equation (\ref{consis1}) then %easily 
follows.

To prove the consistency of $\widehat \beta$, we need the following assumptions :
\begin{itemize}
\item[(AC1)] For $\tau$ appearing in conditions (\ref{hhhh_0})-(\ref{hhhh}),
\begin{equation*}%\label{unif1}
 \sup_{x\in \mathcal{X}} \sup_{t\in (-\infty,\tau]} | \widehat H_k([t,\infty)\mid x) - H_k([t,\infty)\mid x) | = o_{\mathbb{P}}(1),\qquad k\in\{0,1\}.
\end{equation*}

\item[(AC2)] The parameter set $B\subset\mathbb{R}^p$ is compact.

\item[(AC3)] There exist some constants $a>0$ and $c_1>0$ such that
$$
\left| \phi(x,\beta) - \phi(x,\beta^\prime) \right| \leq c_1 \| \beta - \beta^\prime \|^a,\qquad \forall\beta,\beta^\prime\in B,\;\forall x,x^\prime \in\mathcal{X}.
$$

\item[(AC4)] $\inf_{\beta \in B} \inf_{x \in \mathcal{X}} \phi(x,\beta) > 0$. 

%\item[(AC4)]
%There exists some constant $c_2$ such that $
%\inf_{\beta\in B}\inf_{x\in\mathcal{X}} \phi(x,\beta) \geq c_2 >0.
%$
\end{itemize}

Now we can state our consistency result.

\begin{theor}\label{consist_prop}
Assume that (AC1)-(AC4) and (\ref{iden_cd1}), (\ref{iden_cd2}), (\ref{iden_cd3}), (\ref{model_inden}), (\ref{hhhh_0}) and (\ref{hhhh}) hold true. Moreover, assume that there exists a unique value $\beta_0$ in the  parameter space $B$ such that (\ref{def_beta_0}) is true.
%$$
%\phi(X,\beta_0) = 1- \prod_{t\in (0,\infty)} \left\{ 1 - \frac{H_1(dt \mid X)}{H([t,\infty)\mid X)}  \right\},\qquad \text{almost surely}.
%$$
Then,
$
\widehat\beta - \beta_0= o_{\mathbb{P}}(1).
$
\end{theor}

Let us point out that the consistency result is stated in terms of the subdistributions of the observations and the conditional probability model $\{\phi(x,\cdot): \beta \in B\}.$ If the identification assumptions used in Proposition \ref{ident_lik} hold true and the model is correctly specified, $\phi(\widehat\beta,x)$ consistently estimates the cure probability $\mathbb{P}(T=\infty\mid x)$ for all $x$ in the support of $X.$  Let us also notice that condition (AC3) guarantees the Glivenko-Cantelli property for certain classes of functions.  It could be significantly weakened, but in the applications our condition (AC3) will cover the common modeling situations.  Condition (AC4) is a weak condition on the %parametric 
model $\phi(x,\beta)$ and is e.g.\ satisfied for the logistic model if $\mathcal{X}$ and $B$ are compact. 

%Let
%$$
%l(Y_i,X_i,\delta_i;\beta) =  \{ \phi(X_i; \beta) F_{T,\beta}^{(1)}(\{Y_i\}\! \mid\! X_i) \} ^{\delta_i} \{ \phi(X_i;\beta) F_{T,\beta}^{(1)} ([Y_i,\infty)\! \mid \! X_i) + 1 - \phi(X_i;\beta) \}^{1-\delta_i},
%$$
%so that
%$$
%\mathcal{L}_n (\beta) = \prod_{i=1}^n l(Y_i,X_i,\delta_i;\beta).
%$$

%and
%to suppose a bounded support for the covariate vector $X,$  a compact parameter set $B$
%and some regularity on the function $\phi(\cdot,\cdot)$ in such way that
%$$
%\sup_{\beta\in B} \sup_x |\phi(x,\beta) - \phi(x,\beta_0)|
%$$
%could be arbitrarily small if $B$ is fixed sufficiently small.
%
%\color{black}
%******* By the way, we have to decide how to define $\beta_0$ which is no longer a `true value' ;-)  ************
%\color{black}
%
%
%and

\subsection{Asymptotic normality}

For the asymptotic normality we will use the approach in Chen \emph{et al.}  (2003). For this purpose we use the derivative of $\log \widehat L_n(\beta)$ with respect to $\beta$.

%In the derivative it is involved the derivative with respect to $\beta$ of the map
%$$
%\beta\mapsto F_{T,\beta}^{(1)}([y,\infty)\mid x).
%$$
%By the Duhamel identity, see Gill \& Johansen (1999), we have for any $y,x$
%\begin{eqnarray}\label{der_ft}
%&& \!\!\!\!\!\!\!\!\!\!\!\!\!\!\!\!\!\! \nabla_\beta F_{T,\beta}^{(1)}([y,\infty)\mid x)=  - F_{T,\beta}^{(1)}([y,\infty)\mid x)\nabla_\beta \phi(x,\beta)\notag \\
%&&\!\!\!\!\!\!\!\!\times \int_{(0,y)} \frac{F_C ([u,\infty)\mid x)}{1-\Lambda_{T,\beta}^{(1)}(\{u\}\mid x)}\; \frac{H_1(du\mid x)}{\left[  H([u,\infty)\mid x) - (1-\phi(x,\beta)) F_C ([u,\infty)\mid x)\right]^2}.
%\end{eqnarray}
%The following uniform convergence holds for this derivative
%\begin{equation}\label{unif4}
%\sup_{\beta\in B} \sup_x \sup_{y\in [0,\tau]} | \nabla_\beta \widehat F_{T,\beta}^{(1)}([y,\tau]\mid x) - \nabla_\beta F_{T,\beta}^{(1)}([y,\tau]\mid x) | = o_{\mathbb{P}}(1).
%\end{equation}
%
%\qquad
%\color{black}
%TO BE COMPLETED
%\color{black}
%\newpage

First note that the vector of partial derivatives of the log-likelihood $\log \widehat L_n(\beta)$ with respect to the components of $\beta$ equals
\begin{eqnarray*}
\nabla_\beta \log \widehat L_n(\beta) &=& \frac{1}{n} \sum_{i=1}^n \nabla_\beta \Big\{ \frac{}{} \delta_i \log \phi(X_i,\beta) + \delta_i \log[ \widehat H_1(\{Y_i\}\mid X_i)] \\
&& - \delta_i\log\left[ \widehat H([Y_i,\infty)\mid X_i) - (1-\phi(X_i,\beta)) \widehat F_C([Y_i,\infty)\mid X_i) \right] \\
&& + \delta_i \log \widehat F_{T,\beta}^{(1)} ([Y_i,\infty)\mid X_i) %+ \delta_i \log \widehat F_C([Y_i,\infty)\mid X_i)
\\
&& + (1-\delta_i) \log \left[\phi(X_i,\beta) \widehat F_{T,\beta}^{(1)} ([Y_i,\infty)\mid X_i) + 1 - \phi(X_i,\beta)  \right]%\\
%&& +(1-\delta_i) \log \left[\widehat F_C(\{Y_i\}\mid X_i) f_X(X_i)  \right]
\Big\}.\\
&=& \frac{1}{n} \sum_{i=1}^n \nabla_\beta q_i(\beta, \widehat H_0, \widehat H_1),
\end{eqnarray*}
where $q_i$ is defined in the proof of Theorem \ref{consist_prop}.

To develop the asymptotic normality of our estimator $\widehat \beta$, we embed the nuisance functions $H_k([\cdot,\infty) \mid \cdot)$ ($k=0,1$) in a functional space ${\cal H}$, which is equipped with a pseudo-norm $\|\cdot\|_{\cal H}$.  Both the space ${\cal H}$ and its pseudo-norm $\|\cdot\|_{\cal H}$ will be chosen depending on the estimators $\widehat H_k([\cdot,\infty) \mid \cdot)$, $k=0,1$, and have to satisfy certain conditions, which we give below.  The true vector of nuisance functions is
$$
\eta_0(t,x) = (\eta_{01}(t,x), \eta_{02}(t,x)) = (H_0([t,\infty)\mid x) , H_1([t,\infty)\mid x)).
$$
For each $x \in {\cal X}$ and for each $\eta \in {\cal H}$, let $\eta(dt,x) = (\eta_{1}(dt,x), \eta_{2}(dt,x))$ be the measures associated to the non-increasing functions
$\eta(\cdot,x) = (\eta_{1}(\cdot,x), \eta_{2}(\cdot,x))$, and define
\begin{equation}\label{m1}
M_n(\beta,\eta) = \frac{1}{n}\sum_{i=1}^n m(Y_i,\delta_i,X_i;\beta ,\eta)
\end{equation}
and
\begin{equation}\label{M}
M(\beta,\eta) = \mathbb{E}[m(Y,\delta,X;\beta ,\eta)],
\end{equation}
where
\begin{multline*}
m(t,\delta,x;\beta ,\eta) = \frac{\delta \nabla_\beta \phi(x,\beta)}{\phi(x,\beta)} - \frac{\delta \nabla_\beta \phi(x,\beta) T_2 (\eta)(t,x)}{T_1 (\beta,\eta)(t,x)} + \delta T_4 (\beta, \eta)(t,x) \\
+ (1-\delta) \frac{\nabla_\beta \phi(x,\beta) T_3 (\beta,\eta)(t,x) + \phi(x,\beta) (T_3T_4) (\beta,\eta)(t,x) - \nabla_\beta \phi(x,\beta)}{\phi(x,\beta)T_3 (\beta, \eta)(t,x) + 1 - \phi(x,\beta)},
\end{multline*}
where
\begin{eqnarray}
T_1 (\beta,\eta)(t,x) &=& \left[ \eta_{1}(t,x) +  \eta_{2}(t,x) \right] - (1-\phi(x,\beta)) T_2 (\beta,\eta)(t,x)\label{T_1}\\
T_2 (\beta,\eta)(t,x) &\equiv& T_2 (\eta)(t,x) = \prod_{-\infty<s<t}\left( 1 - \frac{\eta_1(ds,x)}{(\eta_1+\eta_2) (s,x)}  \right)  \notag\\
T_3 (\beta,\eta)(t,x) & = & \prod_{-\infty<s<t}\left( 1 - \frac{\eta_2 (ds,x)}{T_1 (\beta,\eta)(s,x)}  \right) \notag\\
T_4 (\beta,\eta)(t,x) & = &\nabla_\beta \log T_3 (\beta,\eta)(t,x)\notag\\
&&= -\nabla_\beta \phi(x,\beta) \int_{(-\infty,t)} \frac{T_2 (\eta)(s,x)\eta_2(ds,x)}{T_1 (\beta,\eta)(s,x)[T_1 (\beta,\eta)(s,x) - \eta_2(\{s\},x)]}. \notag
\end{eqnarray}
Note that for $\eta=\eta_0$, we have
$$T_1(\beta,\eta_0)(t,x) = H([t,\infty)\mid x) - (1-\phi(x,\beta)) F_C([t,\infty)\mid x) ,$$
$T_2(\eta_0)(t,x) = F_C([t,\infty)\mid x)$ and $T_3 (\beta,\eta_0)(t,x) = F_{T,0}^\beta([t,\infty)\mid x).$ Hence, we have that
$$ M(\beta_0,\eta_0)=0 \quad \mbox{ and } \quad \|M_n(\widehat \beta,\widehat \eta)\| = \mbox{inf}_{\beta \in B} \|M_n(\beta,\widehat \eta)\|, $$
where $\widehat \eta(t,x) = (\widehat\eta_1(t,x), \widehat\eta_2(t,x)) = (\widehat H_0([t,\infty)\mid x), \widehat H_1([t,\infty)\mid x))$.

If in addition $\beta=\beta_0$,
$$
T_1(\beta_0,\eta_0)(t,x) = F_{T,0}^{\beta_0}([t,\infty)\mid x) F_C([t,\infty)\mid x), \qquad t\in (-\infty,\tau],
$$
and thus for any $x$, the map $t\mapsto T_1(\beta_0,\eta_0)(t,x)$ is decreasing on $(-\infty,\tau]$. Moreover, by condition (\ref{hhhh}),
$$
\inf_{x\in\mathcal{X}} T_1(\beta_0,\eta_0)(\tau,x) > 0.
$$
We need this lower bound to be valid on a neighborhood around $\beta_0$. Hence, let us consider  $B_0$ a neighborhood of $\beta_0$ such that
\begin{eqnarray} \label{infinf}
\inf_{\beta\in B_0} \inf_{x\in\mathcal{X}} T_1(\beta,\eta_0)(\tau,x) > 0.
\end{eqnarray}
The existence of $B_0$ is guaranteed by condition (\ref{hhhh}) and the regularity of the function $\phi(\cdot,\cdot)$; see assumption (AN3) below that strengthens assumption (AC3). Finally, let us note that 
%condition (\ref{iden_cd2}) implies that 
by construction for any $t \in (-\infty,\tau)$, $H(\{t\} \mid x)=H_0(\{t\} \mid x)+H_1(\{t\} \mid x)\geq (1-\phi(x,\beta)) F_C(\{t\}\mid x)+H_1(\{t\} \mid x)$ and thus 
\begin{multline*}
H([t,\infty)\mid x) - (1-\phi(x,\beta)) F_C([t,\infty)\mid x) - H_1(\{t\} \mid x) \\\geq H((t,\infty)\mid x) - (1-\phi(x,\beta)) F_C((t,\infty)\mid x).
\end{multline*}
Then, by the arguments guaranteeing the existence of a set $B_0$ as in equation (\ref{infinf}),
%and thus \footnote{I don't understand the previous and the next formula.  Au secours ?}
\begin{eqnarray} \label{infinfinf}
\inf_{\beta\in B_0} \inf_{x\in\mathcal{X}} \inf_{t\in(-\infty,\tau)} [T_1(\beta,\eta_0)(t,x) - H_1(\{t\} \mid x)] > 0.
\end{eqnarray}
%\newpage
Further, define the G\^ateaux derivative of $M(\beta,\eta_0)$ in the direction $[\eta-\eta_0]$ by
$$ \nabla_\eta M(\beta,\eta_0)[\eta-\eta_0] = \lim_{\tau \rightarrow 0} \tau^{-1} \Big[M(\beta,\eta_0+\tau(\eta-\eta_0)) - M(\beta,\eta_0) \Big], $$
and in a similar way the G\^ateaux derivatives $\nabla_\eta T_j(\beta,\eta_0)[\eta-\eta_0]$ are defined.

We need the following assumptions :

\begin{itemize}
\item[(AN1)]\label{an1} The matrix $\nabla_\beta M(\beta,\eta)$ exists for $\beta$ in a neighborhood $B_0$ of $\beta_0$ and is continuous in $\beta$ for $\beta=\beta_0$.  Moreover, $\nabla_\beta M(\beta_0,\eta_0)$ is non-singular.

\item[(AN2)]\label{an2} $H_k([\cdot,\infty) \mid \cdot) \in {\cal H}$ for $k=0,1$.

\item[(AN3)]\label{an3} The function $\beta \rightarrow \phi(x,\beta)$ is continuously differentiable for all $x \in {\cal X}$, and the derivative is bounded uniformly in $x \in {\cal X}$ and $\beta \in B_0$.   Moreover, $B_0$ is compact and $\beta_0$ belongs to the interior of $B_0$.

\item[(AN4)]\label{an4} For $k=0,1$,  the estimator $\widehat H_k([\cdot,\infty) \mid \cdot)$ satisfies the following :
\begin{itemize}
\item[$(i)$] $\mathbb P(\widehat H_k([\cdot,\infty) \mid \cdot) \in {\cal H}) \rightarrow 1$
\item[$(ii)$] $\|(\widehat H_k-H_k)([\cdot,\infty) \mid \cdot)\|_{\cal H} = o_{\mathbb P}(n^{-1/4})$
\item[$(iii)$] There exist functions $\Psi_1$ and $\Psi_2$, such that
\begin{eqnarray*}
&& \sum_{k=0}^1\mathbb{E}^*\left[ \psi_{1k}(Y,X) \int_{-\infty<u<Y} \psi_{2k} (u,X)  d\Big( (\widehat H_k-H_k) ([u,\infty)\mid X) \Big) \right] \\
&& \hspace*{2cm} =\frac{1}{n}\sum_{i=1}^n \Psi_1(Y_i,\delta_i,X_i) + R_{1n},
\end{eqnarray*}
and
\begin{eqnarray*}
&\m &\m \hspace*{-0.7cm}\sum_{k,\ell=0}^1\mathbb{E}^*\left[ \psi_{3k}(Y,X) \int_{-\infty<u<Y} \psi_{4k}(u,X) \psi_{5k} \Big((\widehat H_k-H_k) ([u,\infty)\mid X),X \Big) dH_\ell(u\mid X) \right] \\
&\m &\m \hspace*{2cm} =\frac{1}{n}\sum_{i=1}^n \Psi_2(Y_i,\delta_i,X_i) + R_{2n},
\end{eqnarray*}
where $\mathbb{E}^*$ denotes conditional expectation given the data $(Y_i,\delta_i,X_i),$ $1\leq i\leq n$, the functions $\psi_{jk}$ are defined in (\ref{defphijk}) in the Appendix, and where
$$
\mathbb{E}[ \Psi_\ell(Y,\delta,X)] = 0, \quad R_{\ell n}=o_{\mathbb{P}}(n^{-1/2})
$$
($j=1,\ldots,5; k=0,1; \ell=1,2$).
Note that the above expectations are conditionally on the sample and are taken with respect to the generic variables $Y,\delta,X$ which have the same law as the sample.
\end{itemize}

\item[(AN5)]\label{an5}  The class ${\cal H}$ satisfies $\int_0^\infty \sqrt{\log N(\varepsilon,{\cal H},\|\cdot\|_{\cal H})} d\varepsilon < \infty$, where $N(\varepsilon,{\cal H},\|\cdot\|_{\cal H})$ is the $\varepsilon$-covering number of the space ${\cal H}$ with respect to the norm $\|\cdot\|_{\cal H}$, i.e.\ the smallest number of balls of $\|\cdot\|_{\cal H}$-radius $\varepsilon$ needed to cover the space ${\cal H}$.
\end{itemize}

\smallskip

\begin{theor} \label{asno}
Assume that $\widehat \beta - \beta_0 = o_{\mathbb{P}}(1)$ and that (AN1)-(AN5) and (\ref{iden_cd1}), (\ref{iden_cd2}), (\ref{iden_cd3}), (\ref{model_inden}), (\ref{hhhh_0}) and (\ref{hhhh}) hold true.  Then,
$$ n^{1/2} \big(\widehat \beta - \beta_0 \big) \Rightarrow {\cal N}(0,\Omega), $$
where
$$ \Omega = \big\{\nabla_\beta M(\beta_0,\eta_0)\big\}^{-1} V \big\{\nabla_\beta M(\beta_0,\eta_0)\big\}^{-1} $$
and $V = \mbox{Var}\big(m(Y,\delta,X;\beta_0,\eta_0) + \Psi_1(Y,\delta,X) + \Psi_2(Y,\delta,X) \big)$.
\end{theor}

\subsection{Bootstrap consistency} \label{sec_boot}

Although in principle one can use Theorem \ref{asno} above for making inference, the asymptotic variance $\Omega$ has a complicated structure, and the estimation of $\Omega$  would not only be cumbersome, but its precision for small samples could moreover be rather poor.
We continue this section by showing that a bootstrap procedure can be used to estimate the asymptotic variance of $\widehat \beta$, to approximate the whole distribution of $\widehat \beta$ or to construct confidence intervals or test hypotheses regarding $\beta_0$.

Here, we propose to use a naive bootstrap procedure, consisting in drawing triplets $(Y_i^*,\delta_i^*,X_i^*),$ $1\leq i\leq n,$ randomly with replacement from the  data $(Y_i,\delta_i,X_i),$ $1\leq i\leq n$.  Let $\widehat H_k^*$ be the same estimator as $\widehat H_k$ ($k=0,1$) but based on the bootstrap data, and for each $(\beta,\eta)$ let $M_n^*(\beta,\eta) = n^{-1} \sum_{i=1}^n m(Y_i^*,\delta_i^*,X_i^*;\beta,\eta)$.  Define the bootstrap estimator $\widehat \beta^*$ to be any sequence that satisfies
$$ \|M_n^*(\widehat \beta^*,\widehat \eta^*) - M_n(\widehat \beta,\widehat \eta)\| = \mbox{inf}_{\beta \in B} \|M_n^*(\beta,\widehat \eta^*) - M_n(\widehat \beta,\widehat \eta)\|, $$
where $\widehat \eta^*(t,x) = (\widehat\eta_1^*(t,x), \widehat\eta_2^*(t,x)) = (\widehat H_0^*([t,\infty)\mid x), \widehat H_1^*([t,\infty)\mid x))$.

The following result shows that the bootstrap works, in the sense that it allows to recover correctly the distribution of $n^{1/2}(\widehat \beta - \beta_0)$.

%\bigskip

\begin{theor} \label{boot}
Assume that $\widehat \beta - \beta_0 = o_{\mathbb{P}}(1)$ and that (AN1)-(AN5) hold true.   Moreover, assume that $\nabla_\beta M(\beta,\eta)$ is continuous in $\eta$ (with respect to $\|\cdot\|_{\cal H}$) at $(\beta,\eta)=(\beta_0,\eta_0)$, and that (AN4) holds true with $\widehat H_k-H_k$ replaced by $\widehat H_k^*-\widehat H_k$ ($k=0,1$) in ${\mathbb P}^*$-probability.  Then,
$$ \sup_{u\in\mathbb{R}^p} \Big| {\mathbb{P}}^*\big(n^{1/2} (\widehat \beta^* - \widehat \beta) \le u \big) - {\mathbb{P}}\big(n^{1/2} (\widehat \beta - \beta_0) \le u \big) \Big| = o_{\mathbb{P}}(1), $$
where $\mathbb P^*$ denotes probability conditionally on the data $(Y_i,\delta_i,X_i)$, $i=1,\ldots,n$, and where the inequality sign means  the component-wise inequality for vectors.
\end{theor}

\subsection{Verification of the assumptions for kernel estimators} \label{kernel}

We finish this section with an illustration of the verification of the assumptions of our asymptotic results when the conditional subdistributions $H_k$ are estimated by means of kernel smoothing.

Consider the case where $X$ is composed of continuous and discrete components, that is $X = (X_c,X_d)\in\mathcal{X}_c\times \mathcal{X}_d\subset \mathbb{R}^{d_c}\times \mathbb{R}^{d_d},$ with $d_c+d_d  =d\geq 1.$ For simplicity, assume that the support $\mathcal{X}_d$ of the discrete subvector $X_d$ is finite.  We also assume that the life time $T$ has not been transformed by a logarithmic or other transformation, so that its support is $[0,\infty]$. 
The subdistributions  $H_k([t,\infty) \mid x)$ could be estimated by means of a kernel estimator :
$$ \widehat H_k([t,\infty) \mid x) = \sum_{i=1}^n \frac{\widetilde K_{h_n}(X_i-x)}{\sum_{j=1}^n \widetilde K_{h_n}(X_j-x)} I(Y_i \ge t, \delta_i=k), $$
where for any $(x_c,x_d)\in \mathcal{X}_c\times \mathcal{X}_d,$
$$\widetilde K_{h_n}(X_i-x)  = K_{h_n} (X_{c,i}-x_c)  I(X_{d,i} = x_d), $$
$h_n$ is a bandwidth sequence, $K_h(\cdot) = K(\cdot/h)/h^{d_c}$, $K(u) = k(u_1) \cdot \ldots \cdot k(u_{d_c})$ and $k$ is a probability density function.

Nonparametric smoothing of continuous covariates is possible for dimensions $d_c$ larger than 1.  However, the technical arguments necessary to verify the assumptions used for the asymptotic results are tedious. Therefore, in the following we consider $d_c=1$. The discrete covariates do not contribute to the curse of dimensionality, and therefore $d_d$ could be larger than 1. However, for simplicity, below we do not consider discrete covariates. 

To satisfy assumption (AN4),  we need to impose the following conditions :
\begin{itemize}
\item[(C1)] The sequence $h_n$ satisfies $nh_n^4 \rightarrow 0$ and $nh_n^{3+\zeta} (\log n)^{-1} \rightarrow \infty$ for some $\zeta>0$.

\item[(C2)] The support ${\cal X}$ of $X$ is a compact subset of $\mathbb{R}$.

\item[(C3)] The probability density function $K$ has compact support, $\int u K(u) du = 0$ and $K$ is twice continuously differentiable.
\end{itemize}

Further, let ${\cal F}_1$ be the space of functions from $[0,\tau]$ to $[0,1]$  with variation bounded by $M$, and let ${\cal F}_2$ be the space of continuously differentiable functions $f$ from ${\cal X}$ to $[-M,M]$ that satisfy $\sup_{x \in {\cal X}}|f^{\prime}(x)| \le M$ and $\sup_{x_1,x_2 \in {\cal X}} |f^{\prime}(x_1) - f^{\prime}(x_2)| / |x_1-x_2|^\epsilon \le M$ for some $M < \infty$ and $0<\epsilon<1$.  Let
\begin{eqnarray*}
{\cal H} &=& \Big\{(t,x) \rightarrow \eta(t,x) : \eta(\cdot,x) \in {\cal F}_1, \, \frac{\partial}{\partial x} \eta(\cdot,x)  \in {\cal F}_1 \mbox{ for all } x \in {\cal X}, \\
&& \hspace*{3.2cm} \mbox{ and } \eta(t,\cdot) \in {\cal F}_2 \mbox{ for all } 0 \le t \le \tau\Big\}.
\end{eqnarray*}
We define the following norm associated with the space ${\cal H}$ : for $\eta \in {\cal H}$, let
$$ \|\eta\|_{\cal H} = \sup_{0 \le t \le \tau} \sup_{x \in {\cal X}} |\eta(t,x)|. $$

Then, it follows from Propositions 1 and 2 in Akritas and Van Keilegom (2001) that ${\mathbb P}(\widehat H_k \in {\cal H}) \rightarrow 1$ provided $nh_n^{3+\zeta} (\log n)^{-1} \rightarrow \infty$, with $\zeta>0$ as in condition (C1). Moreover, $\sup_{x,t} |\widehat H_k([t,\infty) \mid x) - H_k([t,\infty) \mid x)| = O_{\mathbb P}((nh_n)^{-1/2} (\log n)^{1/2}) = o_{\mathbb P}(n^{-1/4})$ (see Proposition 1 in Akritas and Van Keilegom 2001).  The class ${\cal H}$ satisfies assumption (AN5) thanks to Lemma 6.1 in Lopez (2011).   It remains to show the validity of assumption (AN4)$(iii)$.  We will show the first statement, the second one can be shown in a similar way.  Note that the left hand side equals
\begin{eqnarray*}
&& \sum_{k=0}^1 {\mathbb E} \Big[\psi_{1k}(Y,X) \int_{0<u<Y} \psi_{2k}(u,X) \frac{1}{nh} f_X^{-1}(X) \sum_{i=1}^n K_h(X_i-X) \, d\Big(I(Y_i \ge u, \delta_i=k) \\
&& \hspace*{1.5cm} - H_k([u,\infty) \mid X)\Big) \Big| Y_i,\delta_i,X_i\Big] + o_{\mathbb P}(n^{-1/2}) \\
&& = \sum_{k=0}^1 {\mathbb E} \Big[\psi_{1k}(Y,X) \frac{1}{nh} f_X^{-1}(X) \sum_{i=1}^n K_h(X_i-X) \Big\{- \psi_{2k}(Y_i,X) I(Y_i \le Y,\delta_i=k) \\
&& \hspace*{1.5cm} - \int_{0<u<Y} \psi_{2k}(u,X) \, dH_k([u,\infty) \mid X) \Big\} \Big| Y_i,\delta_i,X_i \Big] + o_{\mathbb P}(n^{-1/2})\\
&& = - \sum_{k=0}^1 n^{-1} \sum_{i=1}^n {\mathbb E} \Big[\psi_{1k}(Y,X_i) \Big\{\psi_{2k}(Y_i,X_i) I(Y_i \le Y,\delta_i=k) \\
&& \hspace*{1.5cm} - \int_{0<u<Y} \psi_{2k}(u,X_i) \, dH_k((-\infty,u] \mid X_i) \Big\} \Big| Y_i,\delta_i,X_i, X=X_i \Big]  + o_{\mathbb P}(n^{-1/2}+h^2),
\end{eqnarray*}
which is of the required form.

\section{Simulations}\label{sec_simul}
%%%%%%%%%%%%%%%%%%%%

In this section we will investigate the small sample performance of our estimation method.   We consider the following model.  The covariate $X$ is generated from a uniform distribution on $[-1,1]$, and the conditional probability $\phi(x,\beta)$ of not being cured follows a logistic model~:
$$ \phi(x,\beta) = \frac{\exp(\beta_1+\beta_2x)}{1+\exp(\beta_1+\beta_2x)}, $$
for any $-1 \le x \le 1$.  We will work with $\beta_0=(\beta_{01},\beta_{02})=(1.75,2)$ and $(1.1,2)$, corresponding to an average cure rate of 20$\%$ respectively 30$\%$.  
The conditional distribution function $F_{T,0}(\cdot|x)$ of the uncured individuals is constructed as follows.   For a given $X$, we draw $T$ from an exponential distribution with mean equal to $\exp[-(\gamma_0+\gamma_1x+\gamma_2/(1+2x^2))]$, where $\gamma_0=\gamma_1=0.5$ and $\gamma_2 \in \{0,1,2\}$.  Next, in order to respect condition (\ref{hhhh}), we truncate this distribution at $\tau$, which is the quantile of order 0.97 of an exponential distribution with mean $\mathbb{E}\{\exp[-(\gamma_0+\gamma_1X+\gamma_2/(1+2X^2))]\}$, i.e.
$$ F_{T,0}([0,t]|x) = 1 - \exp[-\exp(\gamma_0+\gamma_1x+\gamma_2/(1+2x^2)) t] I(0 \le t \le \tau). $$
Note that this is the distribution function corresponding to a Cox model with baseline hazard equal to $I(0 \le t \le \tau)$, and exponential factor equal to $\exp(\gamma_0+\gamma_1x+\gamma_2/(1+2x^2))$.

Next, we generate the censoring variable $C$ independently of $X$ from an exponential distribution with mean equal to 1.65 when $\beta_0=(1.75,2)$, and with mean 1.45 when $\beta_0= (1.1,2)$.  In this way we have respectively 40$\%$ and 50$\%$ of censoring when $\gamma_2=0$.  

In what follows we will compare our estimator of $\beta$ with the estimator proposed by Lu (2008) which assumes a Cox model for the uncured individuals.  The exponential factor in the Cox model is assumed to be linear in the covariate $X$, and hence the Cox model will only be verified when $\gamma_2=0$.    The estimated $\beta$ coefficients under the Cox model are obtained using the \url{R} package \url{smcure}.

For our estimation procedure we used the kernel estimators given in Section \ref{kernel}, and  we programmed $\widehat\beta$ using the optimization procedure \url{optim} in \url{R}.  As starting values we used the estimator obtained from a logistic model based on the censoring indicator (as a surrogate for the unobserved cure indicator).   However, due to the non-concavity of our likelihood function and due to the inconsistency of this vector of starting values, the procedure \url{optim} often ends up in a local maximum instead of the global maximum.  To circumvent this problem, we added the following intermediate step to the estimation procedure.   Based on the initial starting values, we estimate $\beta$ from a logistic model based on the nonparametric estimator $\widehat F_T([0,\infty)|x)$, so we maximize the log-likelihood $\sum_{i=1}^n \{(1-\widehat F_T([0,\infty)|X_i)) \log(\phi(X_i,\beta)) + \widehat F_T([0,\infty)|X_i) \log(1-\phi(X_i,\beta))\}$.  Since this log-likelihood is concave it has a unique local and 
global maximum, %, which is 
expected to %will moreover 
be close to the maximizer of our likelihood.   We now use this intermediate estimate as starting value for our likelihood maximization. % function.   

The results of this two-step maximization procedure are given in Table \ref{table1} for the case where $\beta_0=(1.75,2)$, and in Table \ref{table2} for the case where $\beta_0=(1.1,2)$.  A total of 500 samples of size $n=150$ and $n=300$ are generated, and the tables show the bias and mean squared error (MSE) of the estimators $\widehat\beta_1$ and $\widehat\beta_2$ obtained under the Cox model and from our procedure.  The kernel function $K$ is taken equal to the Epanechnikov kernel : $K(u) = (3/4) (1-u^2) I(|u| \le 1)$.  The bandwidth $h$ of the kernel estimators $\widehat H_k(\cdot|X_i)$ ($k=0,1$) is taken proportional to $n^{-2/7}$ so as to verify regularity condition (C1), i.e.\ $h=cn^{-2/7}$ for several values of $c$, namely $c=2, 3$ and 4.  In addition, we also used the cross-validation (CV) procedure proposed by Li, Lin and Racine (2013) for kernel estimators of conditional distribution functions.  The CV procedure is implemented in the package \url{np} in \url{R}.  For each sample in our simulation, we  calculated these bandwidths for $\widehat H_0$ and $\widehat H_1$ and used the average of these two bandwidths. 

The tables show that our estimator outperforms the one that is based on the Cox model, even when the Cox model is correct.  They also show that our estimator is only mildly sensitive to the bandwidth, which could be explained by the fact that we average out the effect of the bandwidth.  We also see that the CV selection of the bandwidth is working rather well, in the sense that the MSE is close to the smallest value among the MSE's corresponding to the three fixed bandwidths. 

\begin{table}[!h]
\begin{center}
\begin{tabular}{c|c|c|cc|cc|cc|cc|cc}
\hline 
& & & \multicolumn{2}{c|}{$c=2$} & \multicolumn{2}{c|}{$c=3$} & \multicolumn{2}{c|}{$c=4$} & \multicolumn{2}{c|}{$h_{CV}$} & \multicolumn{2}{c}{Cox}\\
\cline{4-13}  
$n$ & $\gamma_2$ & Par.\ & Bias & MSE & Bias & MSE & Bias & MSE & Bias & MSE & Bias & MSE \\
\hline
150 & 0 & $\beta_1$ & .092 & .230 & .012 & .183 &-.075 & .162 & .044 & .224 & .216 & .523 \\
       &    & $\beta_2$ & .002 & .540 &-.231 & .438 &-.506 & .510 & -.158 & .536 & .291 & 1.15 \\
\hline
       & 1 & $\beta_1$ & .063 & .123 &-.021 & .101 &-.094 & .094 & .017 & .120 & .147 & .191 \\
       &    & $\beta_2$ &-.099 & .340 &-.334 & .331 &-.605 & .505 & -.246 & .397 & .278 & .635 \\ 
\hline
       & 2 & $\beta_1$ & .045 & .100 &-.029 & .086 &-.099 & .083 & -.005 & .101 & .124 & .145 \\
       &    & $\beta_2$ &-.109 & .242 &-.356 & .277 &-.632 & .497 & -.302 & .382 & .263 & .476  \\ 
\hline
300 & 0 & $\beta_1$ & .021 & .100 &-.029 & .088 &-.081 & .081 & .004 & .099 & .099 & .139 \\
       &    & $\beta_2$ &-.081 & .266 &-.252 & .268 &-.461 & .363 & -.148 & .288 & .135 & .365  \\
\hline
       & 1 & $\beta_1$ &-.004 & .060 &-.048 & .055 &-.097 & .055 & -.013 & .061 & .097 & .092 \\
       &    & $\beta_2$ &-.107 & .181 &-.278 & .208 &-.482 & .328 & -.150 & .201 & .215 & .302  \\
\hline
       & 2 & $\beta_1$ &-.015 & .050 &-.059 & .046 &-.107 & .049 & -.030 & .052 & .077 & .074  \\
       &    & $\beta_2$ &-.124 & .157 &-.295 & .198 &-.498 & .329 & -.197 & .217 & .181 & .247  \\     
\hline
\end{tabular} 
\vspace*{.3cm}
\caption{{\sl Bias and MSE of $\widehat \beta_1$ and $\widehat \beta_2$ for two sample sizes, three values of $\gamma_2$, three bandwidths of the form $h=cn^{-2/7}$ and the bandwidth $h_{CV}$ obtained from cross-validation. Here, $P(cured) = 0.2$ and $P(censoring) = 0.4$ for $\gamma_2=0$.  The Cox model is satisfied for $\gamma_2=0$.}} 
\label{table1}
\end{center}
\end{table}

\begin{table}[!h]
\begin{center}
\begin{tabular}{c|c|c|cc|cc|cc|cc|cc}
\hline 
& & & \multicolumn{2}{c|}{$c=2$} & \multicolumn{2}{c|}{$c=3$} & \multicolumn{2}{c|}{$c=4$} & \multicolumn{2}{c|}{$h_{CV}$} & \multicolumn{2}{c}{Cox}\\
\cline{4-13}  
$n$ & $\gamma_2$ & Par.\ & Bias & MSE & Bias & MSE & Bias & MSE & Bias & MSE & Bias & MSE \\
\hline
150 & 0 & $\beta_1$ & .082 & .139 & .003 & .116 &-.060 & .105 & .031 & .130 &.116 & .189 \\
       &    & $\beta_2$ &-.017 & .418 &-.244 & .365 &-.525 & .477 &-.217 & .419 & .215 & .618 \\ 
\hline
       & 1 & $\beta_1$ & .057 & .076 &-.011 & .063 &-.068 & .059 & .010 & .071 & .086 & .099 \\ 
       &    & $\beta_2$ &-.107 & .253 &-.328 & .277 &-.608 & .478 &-.294 & .352 & .227 & .434  \\ 
\hline
       & 2 & $\beta_1$ & .049 & .065 &-.016 & .056 &-.071 & .054 & .006 & .066 & .065 & .078 \\
       &    & $\beta_2$ &-.129 & .202 &-.361 & .260 &-.640 & .493 &-.329 & .358 & .196 & .329  \\ 
\hline
300 & 0 & $\beta_1$ & .059 & .074 &-.010 & .060 &-.058 & .056 & .031 & .073 & .060 & .083 \\
       &    & $\beta_2$ &-.068 & .216 &-.237 & .226 &-.453 & .326 &-.153 & .241 & .103 & .257  \\
\hline
       & 1 & $\beta_1$ &-.013 & .036 &-.043 & .034 &-.074 & .035 &-.024 & .035 & .050 & .047  \\
       &    & $\beta_2$ &-.117 & .135 &-.282 & .175 &-.486 & .306 &-.187 & .160 & .159 & .197  \\
\hline
       & 2 & $\beta_1$ &-.049 & .028 &-.080 & .030 &-.037 & .030 &-.034 & .030 & .035 & .037  \\
       &    & $\beta_2$ &-.295 & .168 &-.496 & .306 &-.244 & .194 &-.217 & .182 & .128 & .156 \\     
\hline
\end{tabular} 
\vspace*{.3cm}
\caption{{\sl Bias and MSE of $\widehat \beta_1$ and $\widehat \beta_2$ for two sample sizes, three values of $\gamma_2$, three bandwidths of the form $h=cn^{-2/7}$ and the bandwidth $h_{CV}$ obtained from cross-validation. Here, $P(cured) = 0.3$ and $P(censoring) = 0.5$ for $\gamma_2=0$.  The Cox model is satisfied for $\gamma_2=0$.}} 
\label{table2}
\end{center}
\end{table}

Next, we look at the estimation of the quartiles of the distribution $F_{T,0}(\cdot|x)$ when $x=0.25$.   We estimate these quartiles by means of our nonparametric estimator $\widehat F_{T,0}^{\widehat \beta}(\cdot|x)$ and by means of the Cox model studied in Lu (2008).   The results given in Tables \ref{table3} and \ref{table4} show that, as could be expected, when the Cox model is not satisfied (i.e.\ when $\gamma_2=1$ or 2), the MSE of the quartiles obtained under the Cox model is much higher than the corresponding MSE obtained from our procedure.  This shows the importance of having a model that does not impose any assumptions on the distribution of the uncured individuals and which still provides very accurate estimators for the logistic part of the model.

\begin{table}[!t]
\begin{center}
\vspace*{.3cm}
\begin{tabular}{c|c|c|cc|cc|cc|cc|cc}
\hline 
& & & \multicolumn{2}{c|}{$c=2$} & \multicolumn{2}{c|}{$c=3$} & \multicolumn{2}{c|}{$c=4$} & \multicolumn{2}{c|}{$h_{CV}$} & \multicolumn{2}{c}{Cox}\\
\cline{4-13}  
$n$ & $\gamma_2$ & $p$ & Bias & MSE & Bias & MSE & Bias & MSE & Bias & MSE & Bias & MSE \\
\hline
150 & 0 & $.25$ & .044 & .061 & .025 & .035 & .030 & .032 & .033 & .049 & .031 & .027 \\
       &    & $.50$ & .022 & .049 & .003 & .031 & .003 & .025 & .011 & .043 & .028 & .023 \\
       &    & $.75$ & .024 & .083 & .006 & .055 & .001 & .045 & .011 & .071 & .032 & .038 \\
\hline
       & 1 & $.25$ & .083 & .053 & .102 & .039 & .155 & .053 & .096 & .049 & .267 & .104 \\
       &    & $.50$ & .060 & .051 & .092 & .041 & .144 & .049 & .073 & .048 & .251 & .093 \\
       &    & $.75$ & .072 & .089 & .114 & .075 & .154 & .077 & .091 & .085 & .254 & .117 \\ 
\hline
       & 2 & $.25$ & .126 & .060 & .218 & .082 & .325 & .139 & .202 & .085 & .592 & .401 \\
       &    & $.50$ & .098 & .058 & .189 & .073 & .291 & .121 & .172 & .081 & .513 & .308 \\
       &    & $.75$ & .112 & .107 & .210 & .120 & .296 & .156 & .187 & .124 & .492 & .322 \\
\hline
300 & 0 & $.25$ & .007 & .032 &-.007 & .023 &-.018 & .018 & .002 & .030 & .010 & .013 \\  
       &    & $.50$ &-.001 & .034 &-.013 & .023 &-.026 & .018 &-.006 & .031 & .008 & .013  \\ 
       &    & $.75$ &-.003 & .053 &-.019 & .034 &-.030 & .026 &-.009 & .046 & .010 & .023  \\
\hline
       & 1 & $.25$ & .027 & .026 & .049 & .021 & .081 & .021 & .033 & .026 & .252 & .081 \\
       &    & $.50$ & .028 & .031 & .054 & .025 & .086 & .025 & .034 & .030 & .240 & .076 \\
       &    & $.75$ & .031 & .055 & .053 & .040 & .086 & .038 & .033 & .052 & .235 & .086 \\
\hline
       & 2 & $.25$ & .063 & .031 & .129 & .037 & .216 & .065 & .096 & .039 & .580 & .366 \\
       &    & $.50$ & .055 & .033 & .121 & .040 & .200 & .062 & .086 & .038 & .498 & .274 \\
       &    & $.75$ & .055 & .058 & .109 & .056 & .191 & .074 & .084 & .061 & .452 & .248 \\       
\hline
\end{tabular} 
\vspace*{.3cm}
\caption{{\sl Bias and MSE of the conditional quantiles of order $p=0.25, 0.50$ and $0.75$ at $x=0.25$ for two sample sizes, three bandwidths of the form $h=cn^{-2/7}$ and the bandwidth $h_{CV}$ obtained from cross-validation. Here, $P(cured) = 0.2$ and $P(censoring) = 0.4$ for $\gamma_2=0$.  The Cox model is satisfied for $\gamma_2=0$.}} 
\label{table3}
\end{center}
\end{table}

\begin{table}[!h]
\begin{center}
\vspace*{.3cm}
\begin{tabular}{c|c|c|cc|cc|cc|cc|cc}
\hline 
& & & \multicolumn{2}{c|}{$c=2$} & \multicolumn{2}{c|}{$c=3$} & \multicolumn{2}{c|}{$c=4$} & \multicolumn{2}{c|}{$h_{CV}$} & \multicolumn{2}{c}{Cox}\\
\cline{4-13}  
$n$ & $\gamma_2$ & $p$ & Bias & MSE & Bias & MSE & Bias & MSE & Bias & MSE & Bias & MSE \\
\hline
150 & 0 & $.25$ & .074 & .095 & .020 & .049 & .021 & .044 & .046 & .071 & .039 & .038 \\
       &    & $.50$ & .049 & .069 & .005 & .041 &-.003 & .031 & .023 & .057 & .031 & .029 \\
       &    & $.75$ & .042 & .098 & .007 & .063 &-.005 & .050 & .024 & .084 & .034 & .046 \\
\hline
       & 1 & $.25$ & .107 & .072 & .116 & .053 & .151 & .059 & .119 & .066 & .274 & .116 \\
       &    & $.50$ & .076 & .060 & .103 & .049 & .142 & .054 & .088 & .057 & .255 & .100 \\
       &    & $.75$ & .099 & .113 & .118 & .083 & .154 & .086 & .109 & .102 & .257 & .126 \\ 
\hline
       & 2 & $.25$ & .159 & .085 & .232 & .097 & .322 & .144 & .233 & .109 & .586 & .402 \\
       &    & $.50$ & .120 & .066 & .211 & .089 & .296 & .131 & .195 & .092 & .508 & .307 \\
       &    & $.75$ & .143 & .135 & .228 & .142 & .308 & .178 & .219 & .153 & .489 & .325 \\
\hline
300 & 0 & $.25$ & .032 & .042 &-.005 & .029 &-.029 & .023 & .016 & .038 & .018 & .017 \\  
       &    & $.50$ & .018 & .039 &-.011 & .026 &-.035 & .021 & .007 & .034 & .013 & .016 \\ 
       &    & $.75$ & .015 & .060 &-.018 & .037 &-.037 & .028 &-.001 & .054 & .016 & .022 \\
\hline
       & 1 & $.25$ & .026 & .033 & .046 & .025 & .071 & .023 & .034 & .031 & .252 & .083 \\
       &    & $.50$ & .030 & .036 & .052 & .027 & .080 & .026 & .037 & .034 & .234 & .074 \\
       &    & $.75$ & .033 & .060 & .050 & .042 & .081 & .040 & .038 & .055 & .232 & .087 \\
\hline
       & 2 & $.25$ & .070 & .040 & .132 & .042 & .212 & .064 & .103 & .046 & .570 & .357 \\
       &    & $.50$ & .060 & .038 & .125 & .044 & .201 & .064 & .089 & .043 & .492 & .269  \\
       &    & $.75$ & .062 & .066 & .115 & .061 & .193 & .078 & .090 & .070 & .451 & .251 \\       
\hline
\end{tabular} 
\vspace*{.3cm}
\caption{{\sl Bias and MSE of the conditional quantiles of order $p=0.25, 0.50$ and $0.75$ at $x=0.25$ for two sample sizes, three bandwidths of the form $h=cn^{-2/7}$ and the bandwidth $h_{CV}$ obtained from cross-validation. Here, $P(cured) = 0.3$ and $P(censoring) = 0.5$ for $\gamma_2=0$.  The Cox model is satisfied for $\gamma_2=0$.}} 
\label{table4}
\end{center}
\end{table}

We also verify how close the distributions of $\widehat\beta_1$ and $\widehat\beta_2$ are to a normal distribution.  We know thanks to Theorem \ref{asno} that the estimators converge to a normal limit when $n$ tends to infinity.  Figure \ref{fig1} shows that for $n=150$ the distribution is rather close to a normal limit, especially for $\widehat\beta_1$.   The figure is based on 1000 samples generated from the above model with $P(cured)=0.2$ and $P(censoring)=0.4$.  The results for $n=300$ (not shown here for space constraints) are close to a straight line, showing that the results improve when $n$ increases.

\begin{figure}[!h]
\begin{center}
\scalebox{0.45}{\includegraphics{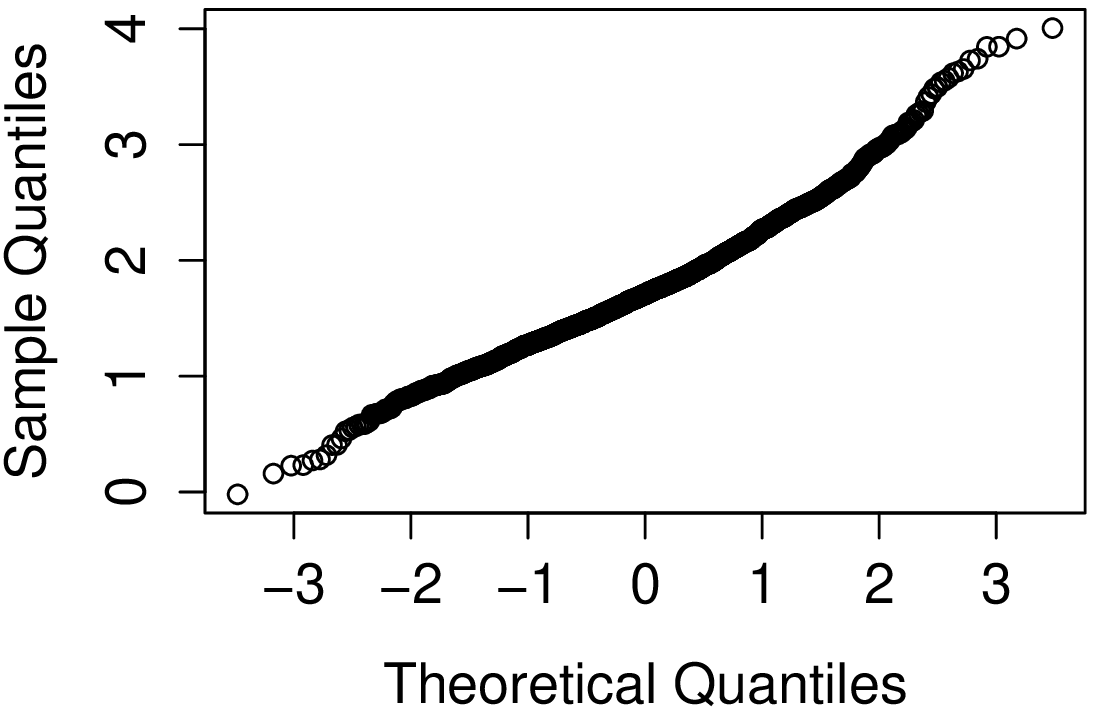}}\;\scalebox{0.45}{\includegraphics{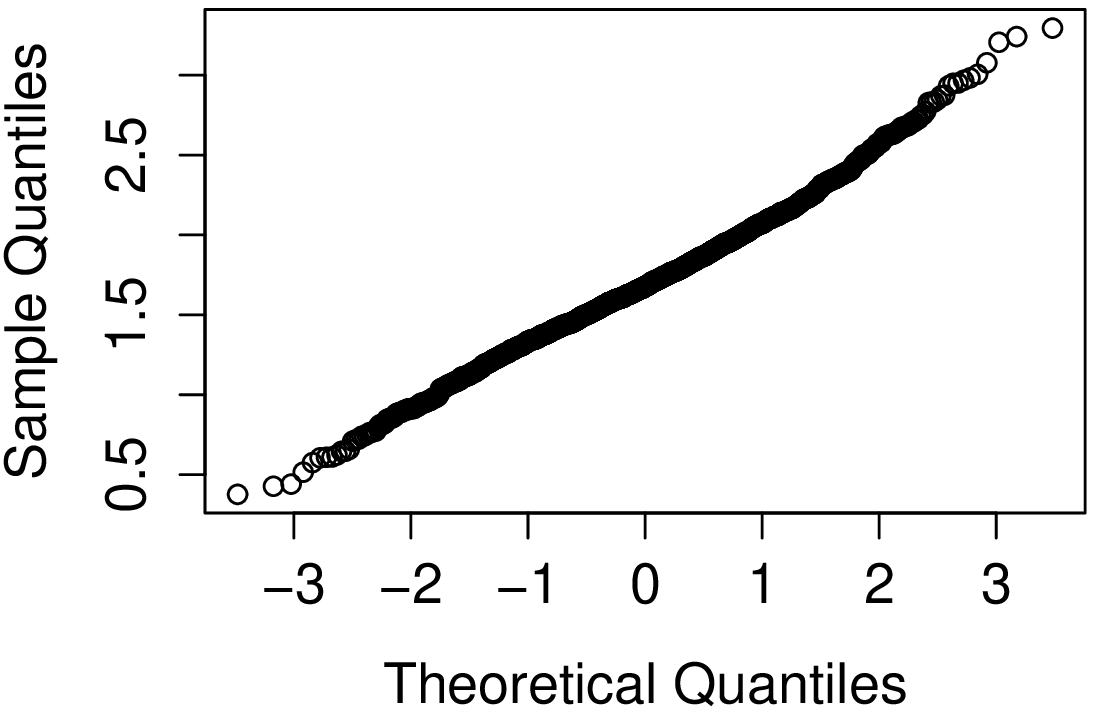}}\;\scalebox{0.45}{\includegraphics{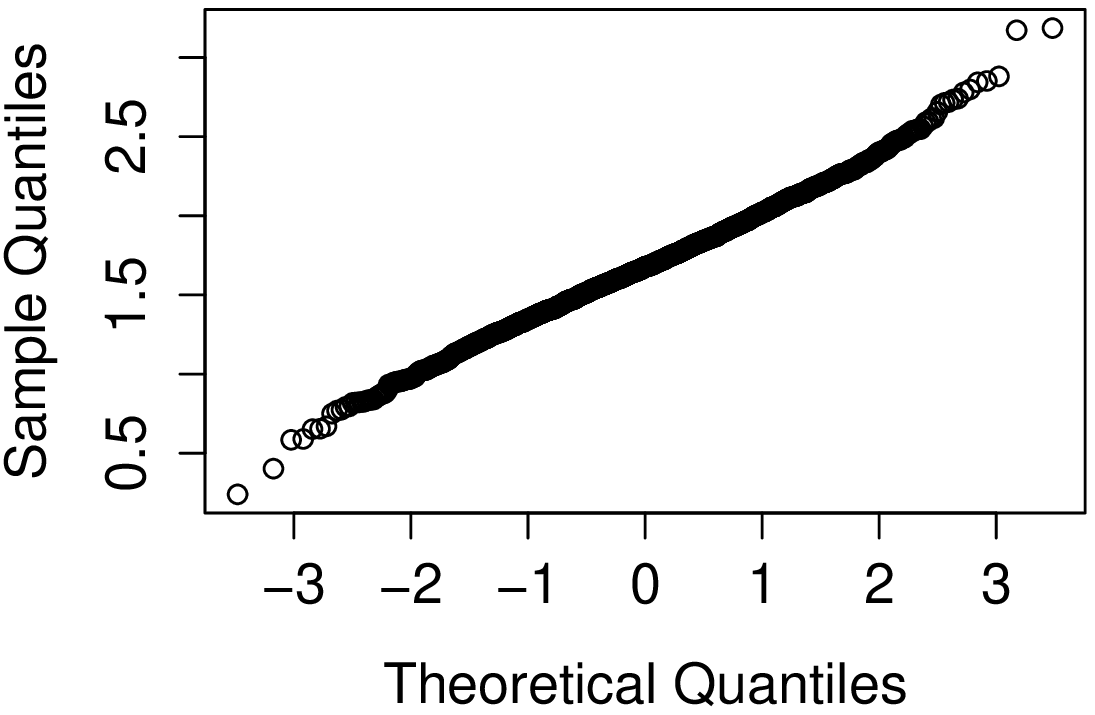}} \\
\scalebox{0.45}{\includegraphics{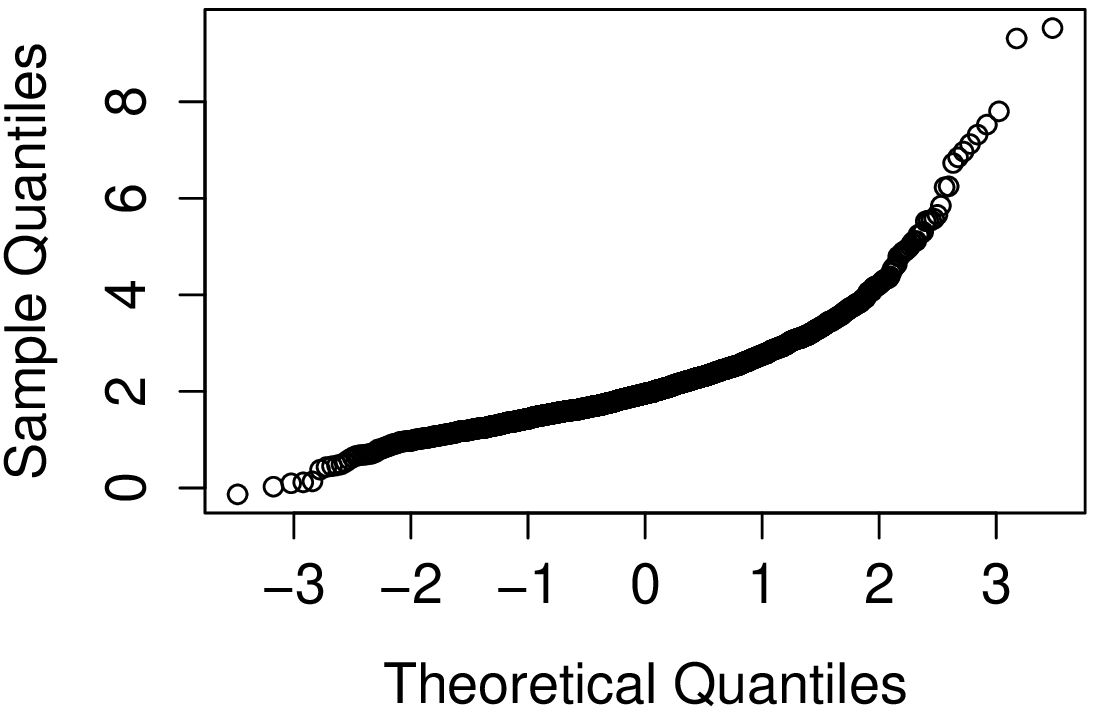}}\;\scalebox{0.45}{\includegraphics{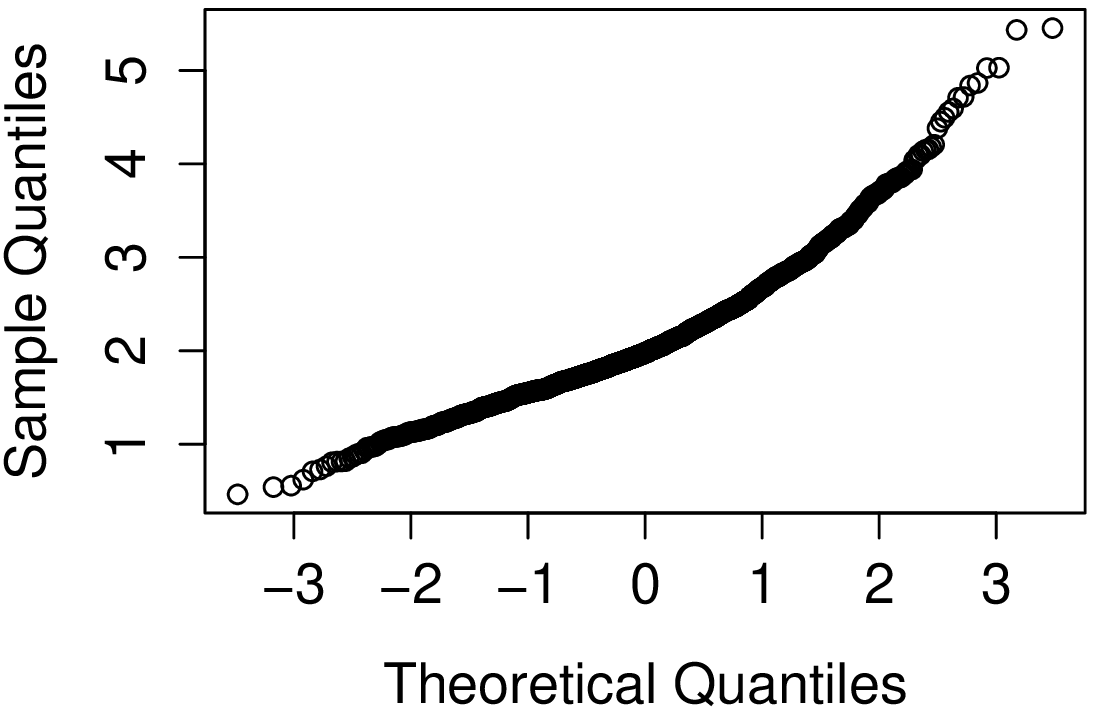}}\;\scalebox{0.45}{\includegraphics{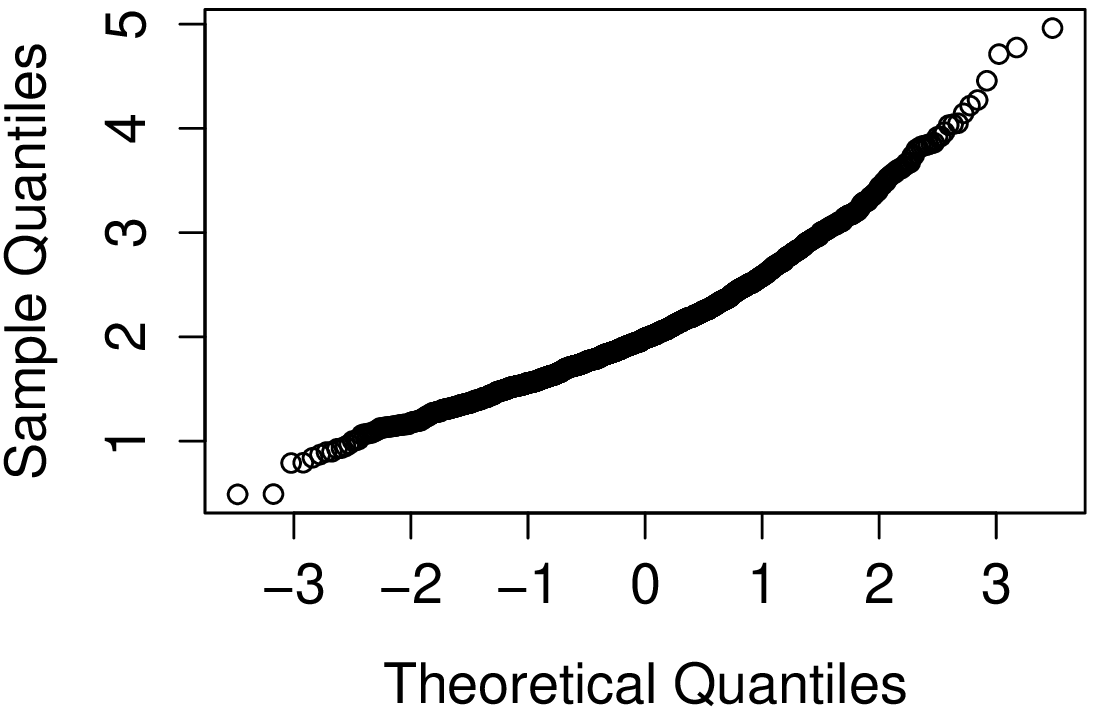}} \\
%\scalebox{0.45}{\includegraphics{qq300beta1g0.eps}}\;\scalebox{0.45}{\includegraphics{qq300beta1g1.eps}}\;\scalebox{0.45}{\includegraphics{qq300beta1g2.eps}} \\
%\scalebox{0.45}{\includegraphics{qq300beta2g0.eps}}\;\scalebox{0.45}{\includegraphics{qq300beta2g1.eps}}\;\scalebox{0.45}{\includegraphics{qq300beta2g2.eps}} \\
\caption{{\it QQ-plots of $\widehat \beta_1$ (first row) and $\widehat \beta_2$ (second row) for 1000 samples of size $n=150$.  The first column corresponds to $\gamma_2=0$, the second to $\gamma_2=1$ and the third to $\gamma_2=2$.   The bandwidth is $h = 3 n^{-2/7}$.}}
\label{fig1}
\end{center}
\end{figure}

Finally, we verify the accuracy of the naive bootstrap  proposed in Section \ref{sec_boot}.  We consider the above model, but restrict attention to $n=150$ and to the case where $P(cured)=0.2$ and $P(censoring)=0.4$.  Figure \ref{fig2} shows boxplots of the variance of $\widehat \beta_1$ and $\widehat \beta_2$ obtained from 250 bootstrap resamples for each of 500 samples.  The bandwidth is $h = 3 n^{-2/7}$.  The empirical variance of the 500 estimators of $\beta_1$ and $\beta_2$ is also added, and shows that the bootstrap variance is well centered around the corresponding empirical variance.  

\begin{figure}[!h]
\begin{center}
\hspace*{-.4cm}\scalebox{0.37}{\includegraphics{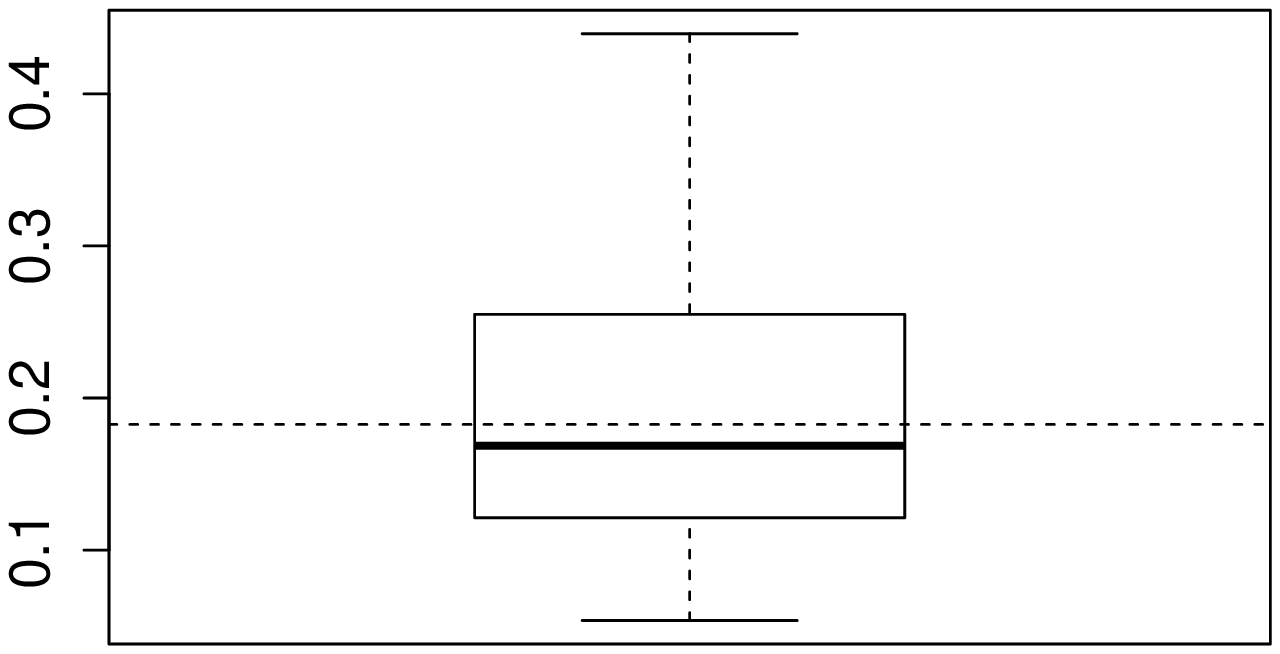}}\;\scalebox{0.37}{\includegraphics{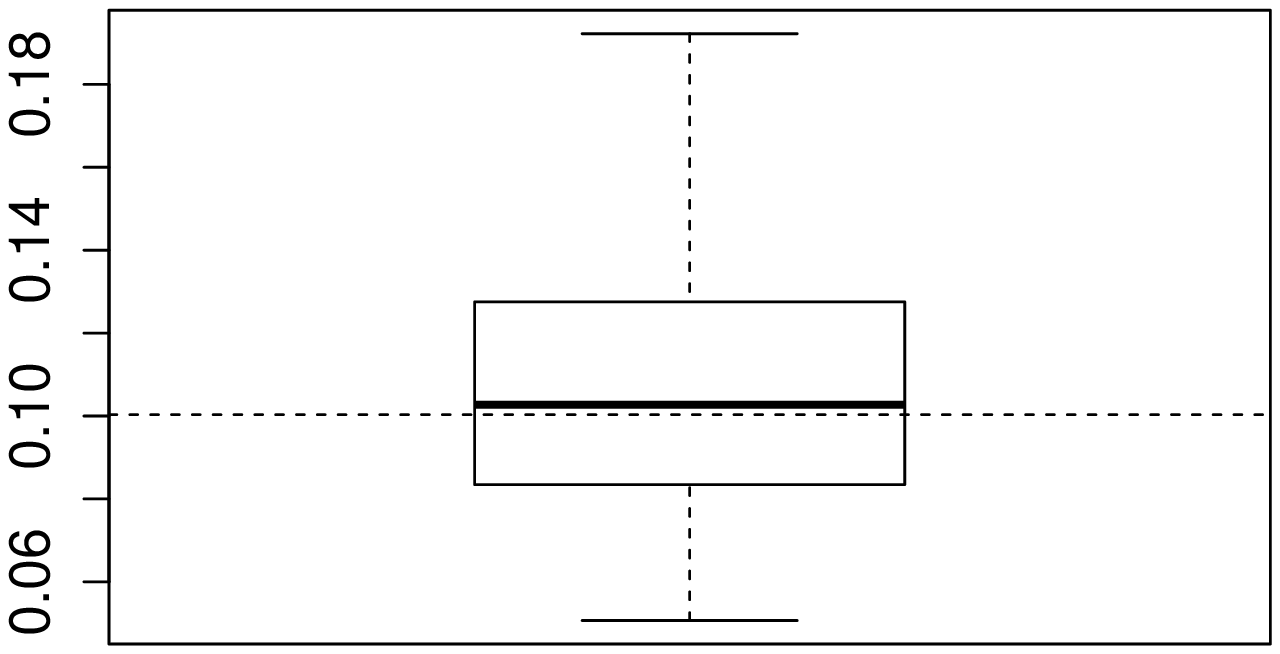}}\;\scalebox{0.37}{\includegraphics{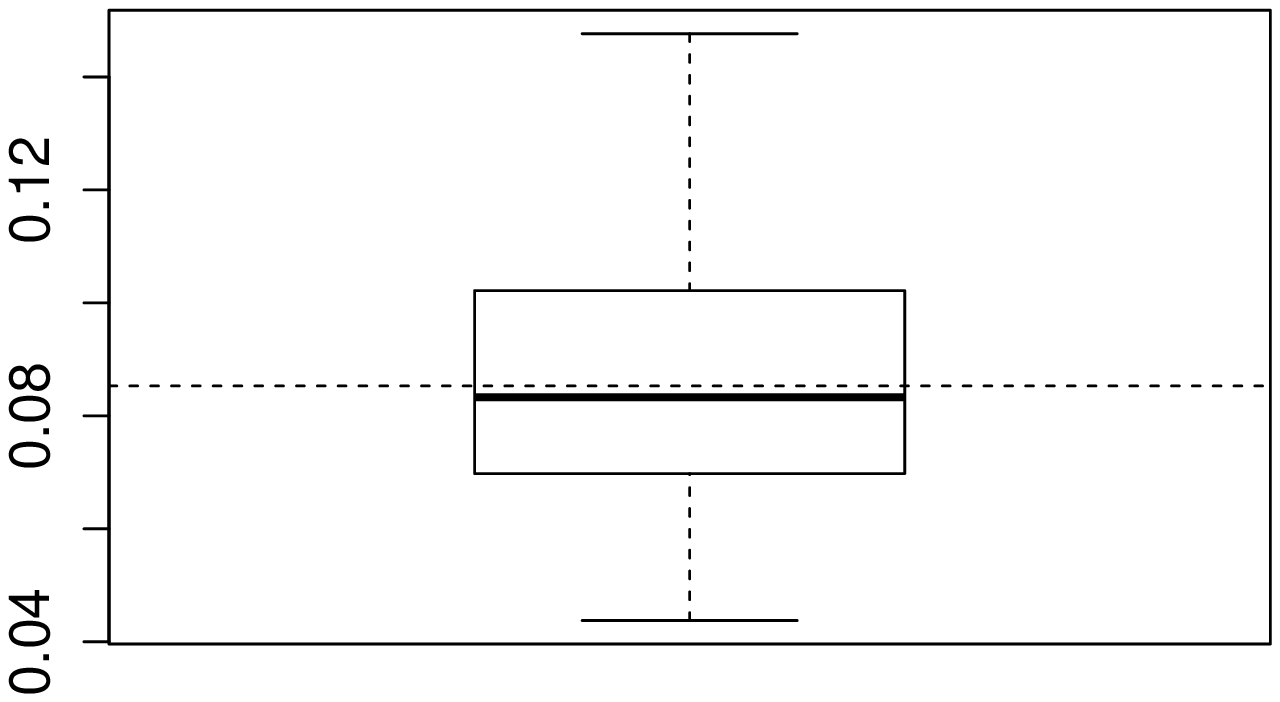}} \\
%\vspace*{-.4cm}
\hspace*{-.4cm}\scalebox{0.37}{\includegraphics{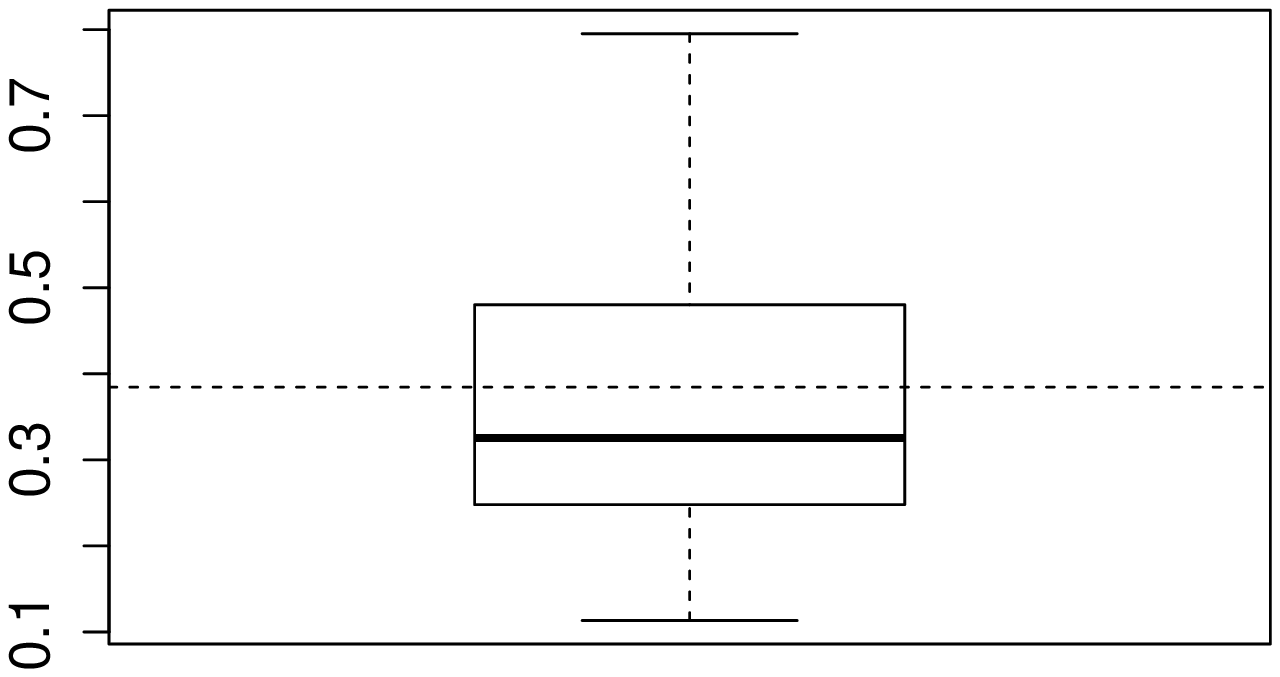}}\;\scalebox{0.37}{\includegraphics{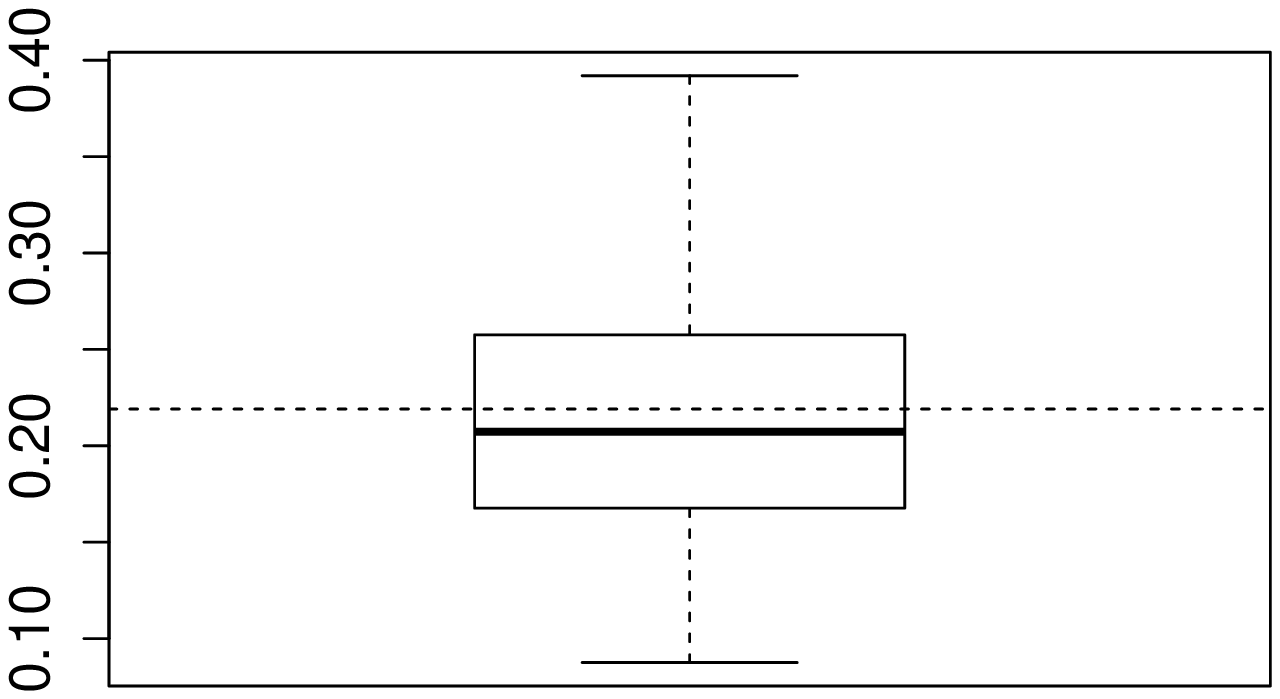}}\;\scalebox{0.37}{\includegraphics{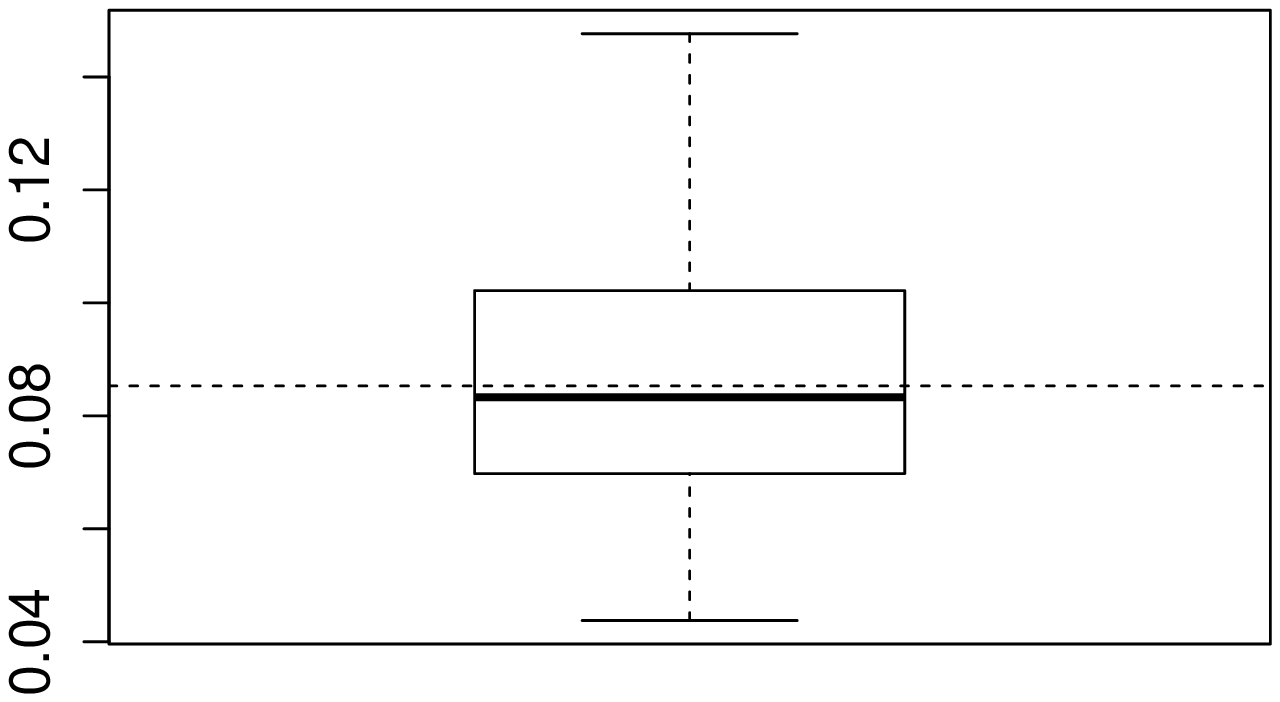}} \\
%\vspace*{-.4cm}
\caption{{\it Boxplots of the variance of $\widehat \beta_1$ (first row) and $\widehat \beta_2$ (second row) obtained from 250 bootstrap resamples for each of 500 samples of size $n=150$.  The first column corresponds to $\gamma_2=0$, the second to $\gamma_2=1$ and the third to $\gamma_2=2$.   The bandwidth is $h = 3 n^{-2/7}$.  The empirical variance of the 500 estimators of $\beta_1$ and $\beta_2$ is also added (dashed line).}}
\label{fig2}
\end{center}
\end{figure}

%%%%%%%%%%%%%%%%%%%%
\section{Data analysis}
%%%%%%%%%%%%%%%%%%%%
\label{sec_data}

Let us now apply our estimation procedure on two medical data sets.  The first one is about 286
breast cancer patients with lymph-node-negative breast cancer treated between 1980 and 1995 (Wang \emph{et al.} (2005)).  The event of interest is distant-metastasis, and the associated survival time is the distant metastasis-free survival time (defined as the time to first distant progression or death, whichever comes first).  107 of the 286 patients experience a relapse from breast cancer. 
The plot of the Kaplan-Meier estimator of the data is given in Figure \ref{fig3}(a) and shows a large plateau at about 0.60. Furthermore, a large proportion of the censored observations is in the plateau, which suggests that a cure model is appropriate for these data.  As a covariate we use the age of the patients, which ranges from 26 to 83 years and the average age is about 54 years. 
%standardised so as to have mean zero and variance 1.  

\begin{figure}[!h]
\begin{center}
\scalebox{0.55}{\includegraphics{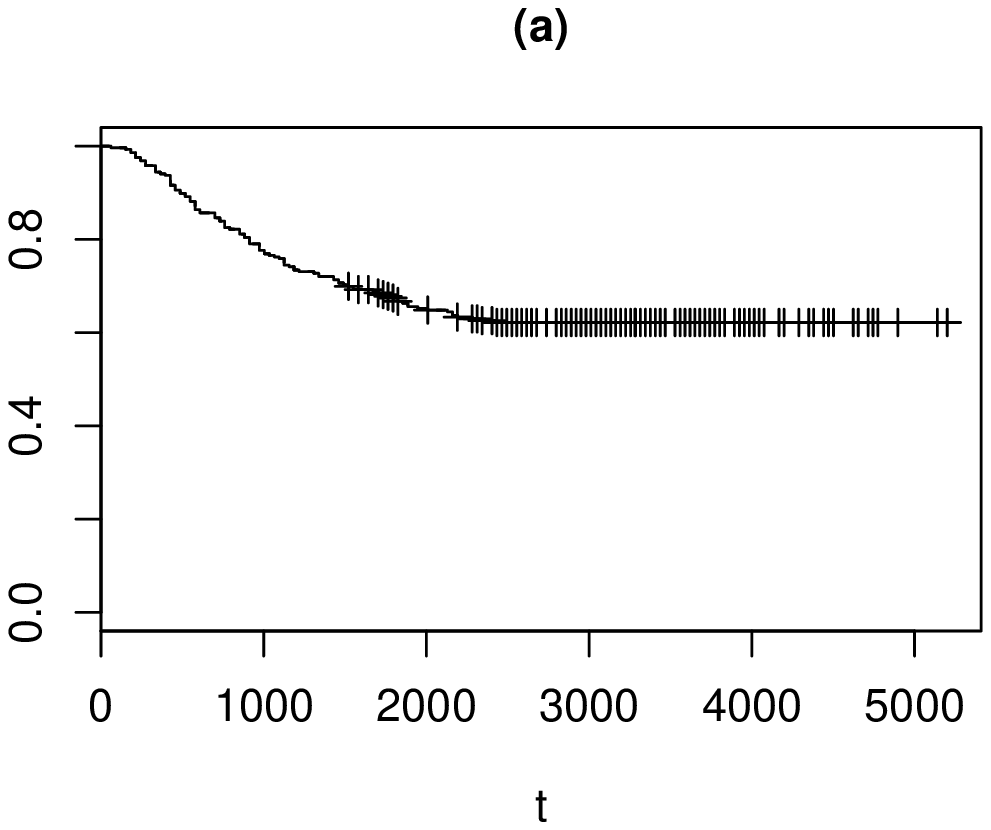}}\;\scalebox{0.55}{\includegraphics{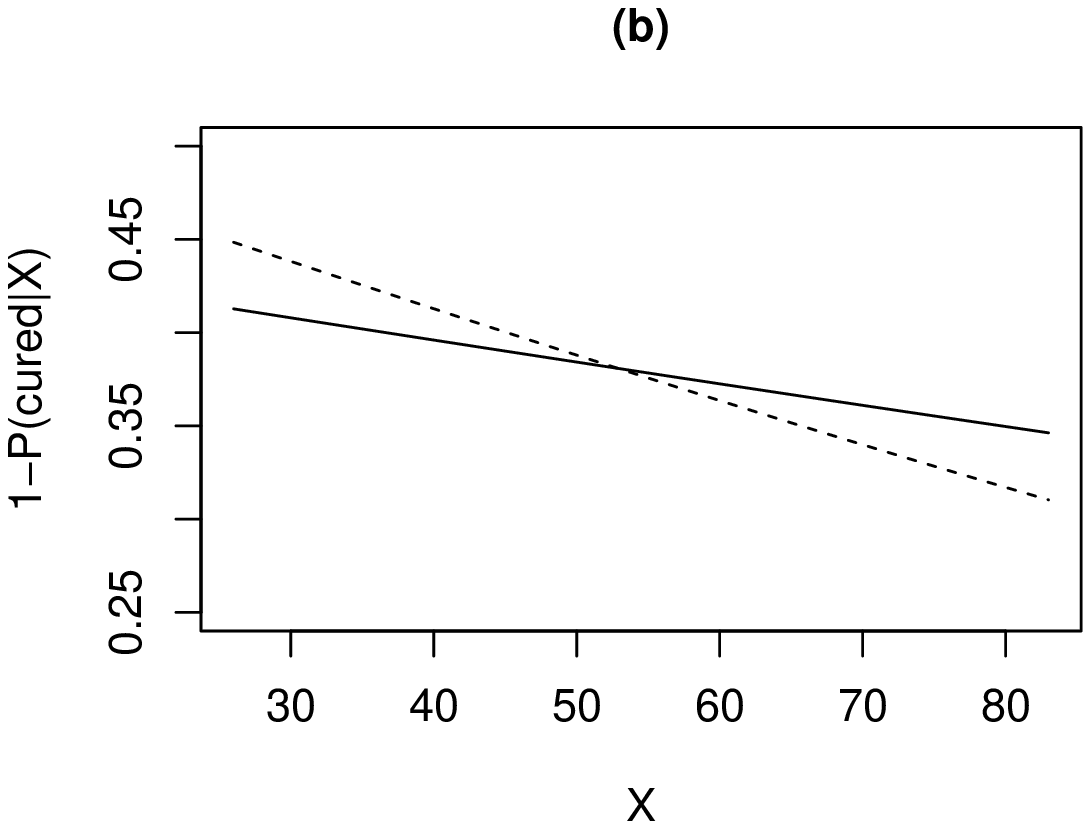}}\\
\vspace*{-.4cm}
\scalebox{0.55}{\includegraphics{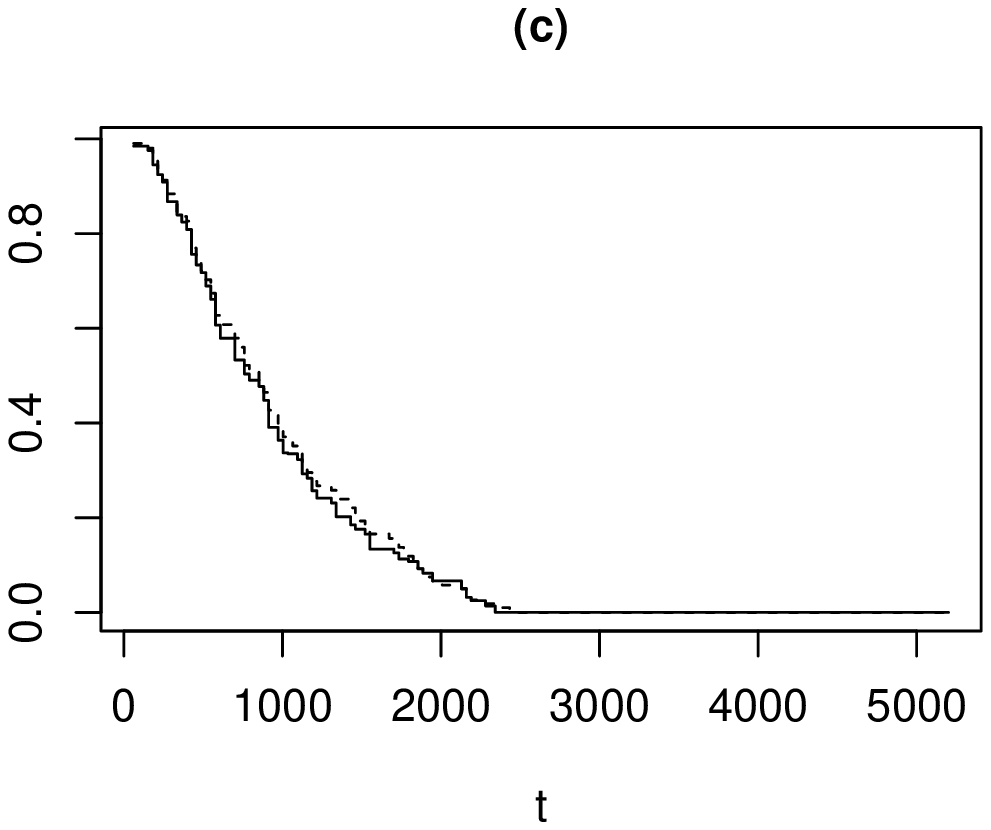}}\;\scalebox{0.55}{\includegraphics{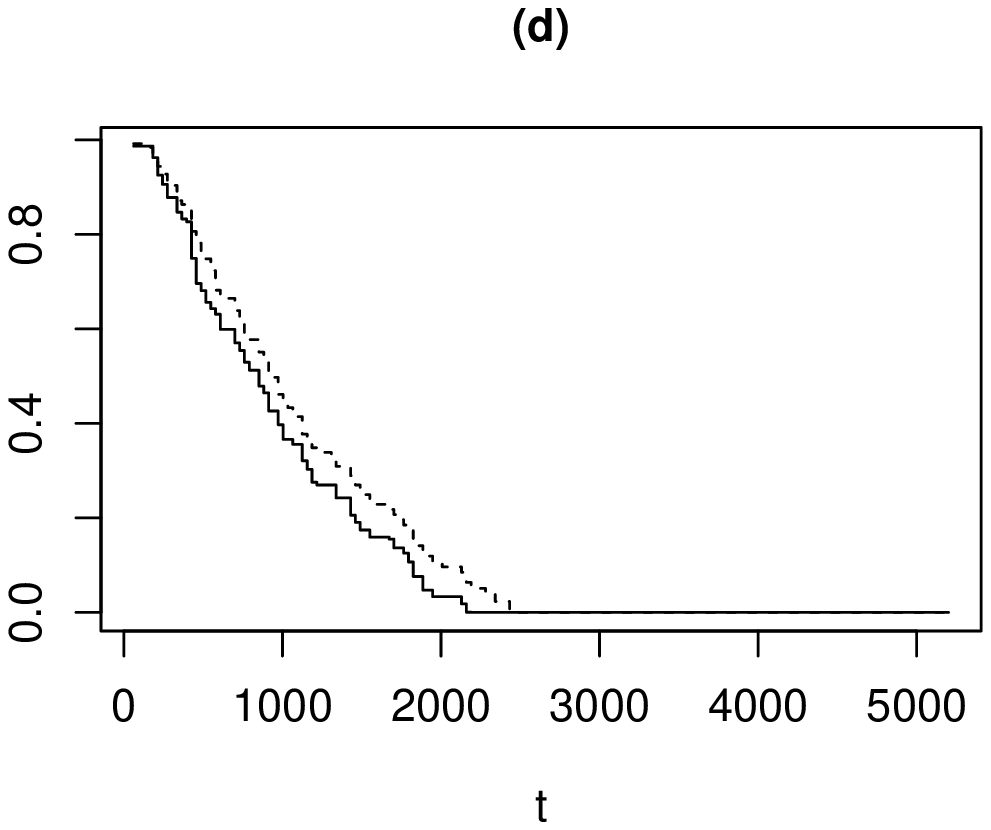}}\\
\vspace*{-.5cm}
\caption{{\it Analysis of the breast cancer data : (a) Kaplan-Meier estimator; 
(b) Graph of the proposed estimator of $\phi(x)$ (solid curve) and of the estimator based on the Cox model (dashed curve); (c) Estimation of $1-F_{T,0}(\cdot|x)$ using the proposed estimator (solid curve) and using the estimator based on the Cox model (dashed curve) when $x=48$; (d) Idem when $x=60$.}}
\label{fig3}
\end{center}
\end{figure}

We estimate $\beta$ using our estimator and using the estimator based on the Cox model.   The bandwidth $h$ is selected using cross-validation, as in the simulation section.   The estimated intercept is -0.224 (with standard deviation equal to 0.447 obtained using a naive bootstrap procedure), and the estimated slope parameter is -0.005  (with standard deviation equal to 0.008).  Under the Cox model the estimated intercept and slope are respectively 0.063 and -0.010.   A 95$\%$ confidence interval  is given by $(-1.100, 0.653)$ for the intercept and $(-0.021, 0.011)$ for the slope, where the variance is again based on the naive bootstrap procedure.   The graph of the two estimators of the function $\phi(x)$ is given in Figure \ref{fig3}(b).   The estimated coefficients and curves are quite close to each other, suggesting that the Cox model might be valid.  This is also confirmed by Figure \ref{fig3}(c)-(d), which shows the estimation of the survival function $1-F_{T,0}(\cdot|x)$ of the uncured patients for $x=48$ and $x=60$ based on our estimation procedure and the procedure based on the Cox model.   The figure shows that the two estimators are close for both values of $x$.  

Next, we analyse data provided by the Medical Birth Registry of Norway (see \linebreak http://folk.uio.no/borgan/abg-2008/data/data.html).  The data set contains information on births in Norway since 1967, related to a total of 53,558 women. We are interested in the time between the birth of the first and the second child, for those mothers whose first child died within the first year ($n=262$).
The covariate of interest is age ($X$), which is the age of the mother at the birth of the first child. 
The age ranges from 16.8 to 29.8 years, with an average of 23.2 years.
%is standardised so that the mean equals zero and the variance equals 1.  
The cure rate is the fraction of women who gave birth only once.  Figure \ref{fig4}(a) shows the Kaplan-Meier estimator, %of the data, 
and suggests that a cure fraction is present.   

\begin{figure}[!b]
\begin{center}
\scalebox{0.55}{\includegraphics{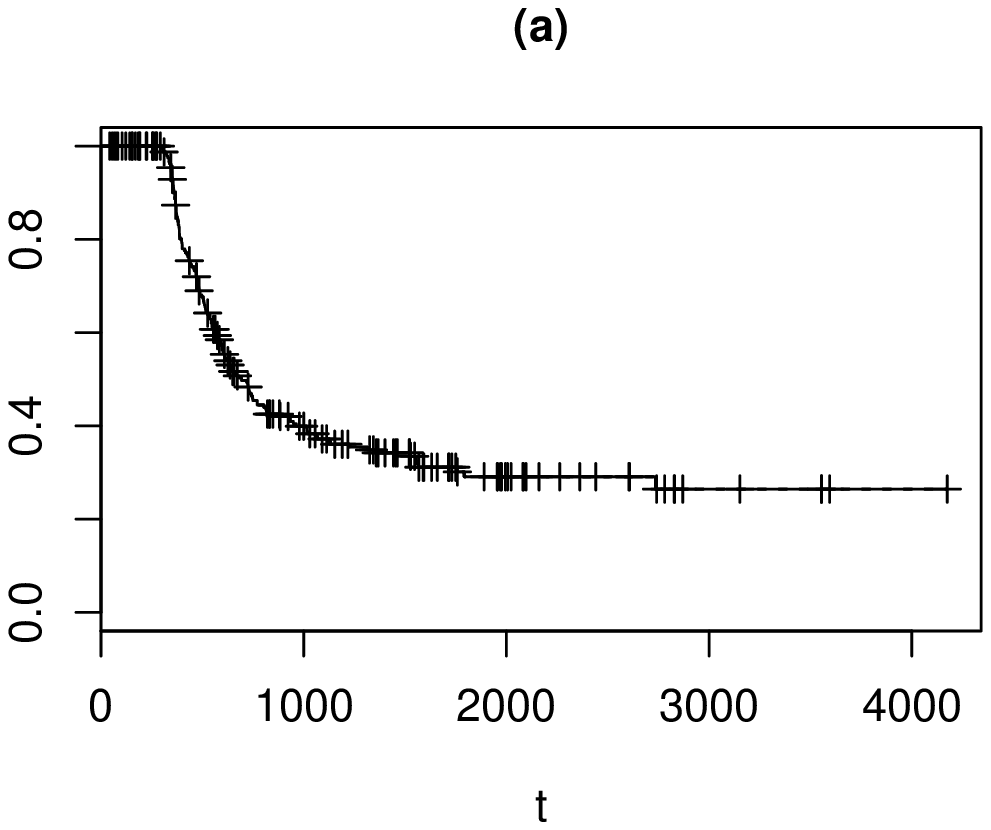}}\;\scalebox{0.55}{\includegraphics{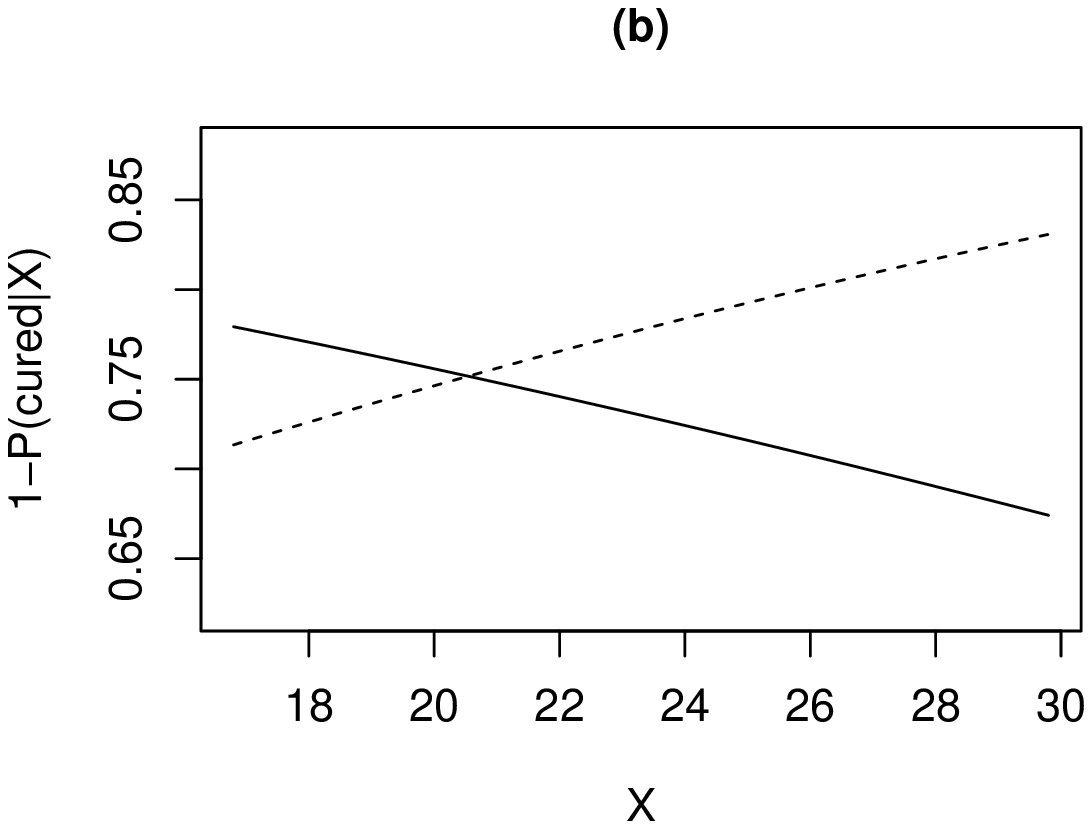}}\\
\vspace*{-.4cm}
\scalebox{0.55}{\includegraphics{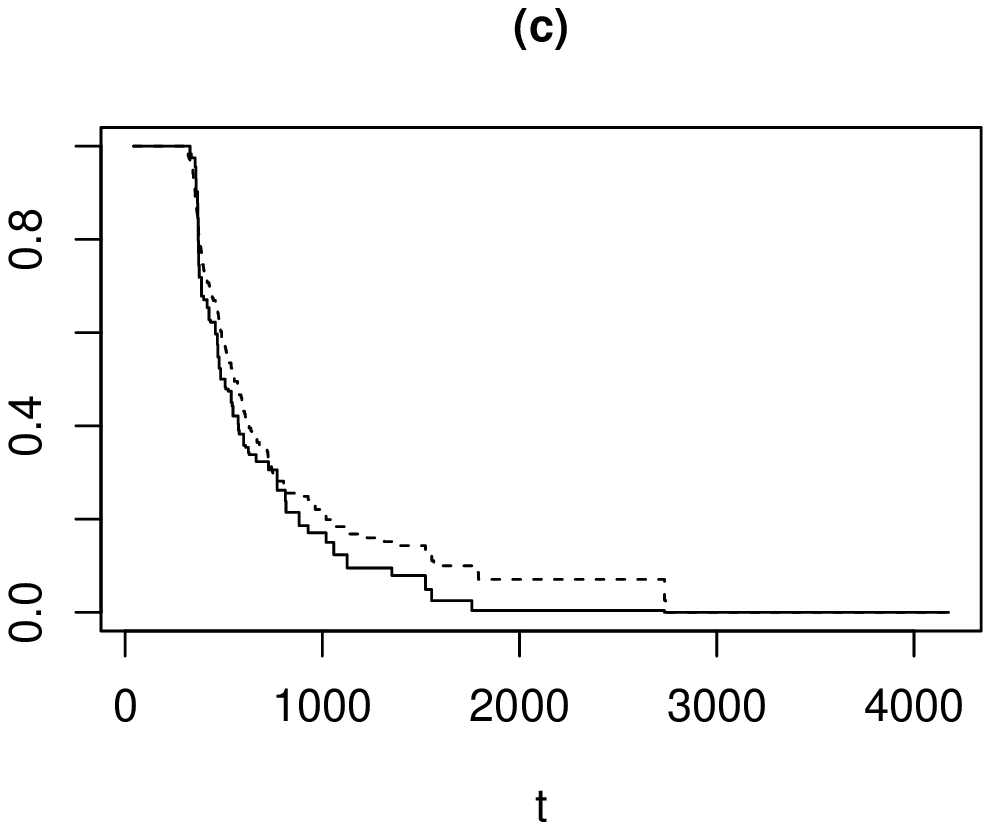}}\;\scalebox{0.55}{\includegraphics{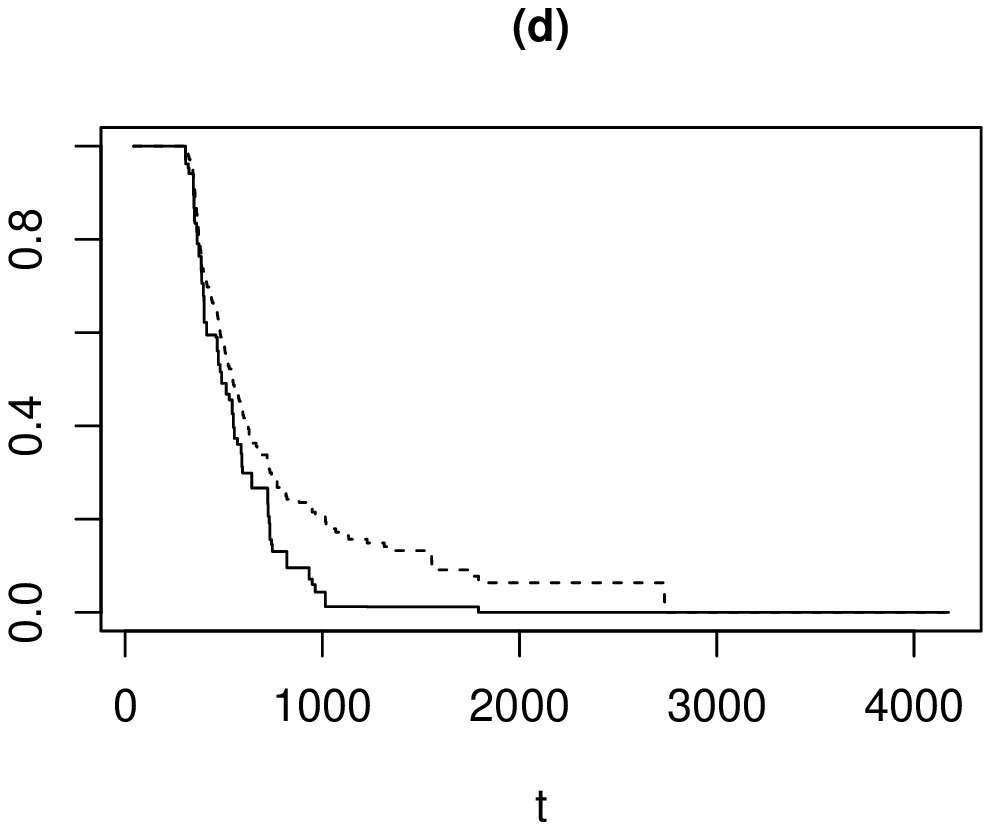}}\\
\vspace*{-.5cm}
\caption{{\it Analysis of the second birth data : (a) Kaplan-Meier estimator; 
(b) Graph of the proposed estimator of $\phi(x)$ (solid curve) and of the estimator based on the Cox model (dashed curve); (c) Estimation of $1-F_{T,0}(\cdot|x)$ using the proposed estimator (solid curve) and using the estimator based on the Cox model (dashed curve) when $x=21$; (d) Idem when $x=25$.}}
\label{fig4}
\end{center}
\end{figure}

As we did for the first data set, we analyse these data using the approach proposed in this paper, and also using the Cox mixture cure model.   The estimated intercept equals 1.952 using our model and 0.034 using the Cox model.  The bootstrap confidence interval for the intercept is $(-1.577,5.481)$ (the estimated standard deviation equals 1.801).   The estimated slope equals -0.041 respectively 0.052 using the two models.  For our estimation procedure the confidence interval is given by $(-0.193,0.111)$ (with estimated standard deviation equal to 0.078).   Figure \ref{fig4}(b) shows that the two estimators of the function $\phi(x)$ are quite different, and have opposite slopes.  Moreover, the survival function $1-F_{T,0}(\cdot|x)$ of the uncured patients is given in Figure \ref{fig4}(c)-(d) for $x=21$ and $x=25$.   We see that the estimator based on the Cox model is quite different from ours, suggesting that the Cox model might not be valid for these data, although a formal test would need to confirm this.  This is however beyond the scope of this paper.  Also note that the estimator of the cure proportion $1-\phi(x)$ is increasing under our model and decreasing under the Cox model.  It seems however natural to believe that the probability of having no second child (so the cure proportion) is increasing with age, which is again an indication that the Cox model is not valid for these data.

%%%%%%%%%%%%%%%%%%%%
\section{Appendix : Proofs}
%%%%%%%%%%%%%%%%%%%%

\begin{proof}[Proof of Proposition \ref{ident_lik}.]
By the properties of the likelihood of a Bernoulli random variable given that $Y\in dt$ and $X=x$, we have
\bqa\label{KL_ineg1}
&& \hspace*{-.5cm} \mathbb{E}\left[\delta \log \frac{\phi(x,\beta) F_{T,0}^\beta(dt\mid x) F_C ((t,\infty)\mid x) }{H_1(dt\mid x)} \right. \\
&& \left.\left.+ (1-\delta) \log \frac{F_C (dt\mid x)\left[\phi(x,\beta)F_{T,0}^\beta ((t,\infty)\mid x) + 1 - \phi(x,\beta)\right] }{H_0(dt\mid x)}\; \, \right| \, Y\in dt, X=x\right] 
\leq   0.\notag
\eqa
Integrate with respect to $Y$ and $X$ and deduce that
$$
\mathbb{E}\left[%\textbf{1}\{X\in\mathcal{X}\}
%\int_{[0,\infty)}
\log p (Y,\delta,X;\beta) \right] \leq \mathbb{E}\left[
%\textbf{1}\{X\in\mathcal{X}\}
%\int_{[0,\infty)}
\log p (Y,\delta,X;\beta_0) \right].
$$
If there exists some $\beta\neq \beta_0$ such that the last inequality becomes an equality, then necessarily $p_1(Y,X;\beta)  =1 $ almost surely. Then
$$
F_C((t,\infty) \mid x)  \phi(x,\beta) F_{T,0}^\beta (dt\mid x) =  F_C((t,\infty) \mid x)  \phi(x,\beta_0) F_{T,0}^{\beta_0}(dt\mid x), \;\; \forall -\infty < t \leq \tau_H(x),
$$
for almost all $x\in\mathcal{X}.$ By Theorem \ref{th_ident}, we deduce that necessarily  $\beta=\beta_0$.
\end{proof}
%}

%\quad
%\smallskip

\begin{lem}\label{F_C_et_al}
Let conditions (\ref{hhhh}), (AC1) and (AC4) hold true. Then,
$$
\sup_{x\in \mathcal{X}}\; \sup_{t\in (-\infty,\tau]} |\widehat F_C ([t,\infty)\mid x) -
F_C ([t,\infty)\mid x)| = o_{\mathbb{P}}(1),
$$
$$
\sup_{\beta\in B} \; \sup_{x\in \mathcal{X}}\; \sup_{t\in (-\infty,\tau]} |T_1(\beta, \widehat H_0,\widehat H_1)(t,x) - T_1(\beta, H_0,H_1)(t,x)|= o_{\mathbb{P}}(1),
$$
where
$$
T_1(\beta, H_0,H_1)(t,x) = H([t,\infty)\mid x) - (1-\phi(x,\beta))F_C ([t,\infty)\mid x),
$$
and $T_1(\beta, \widehat H_0, \widehat H_1)$ is defined similarly, but with $H_0,$ $H_1$ and $F_C $ replaced by
$\widehat H_0,$ $\widehat H_1$ and $\widehat F_C ,$ respectively.
Moreover,
$$
\sup_{\beta\in B} \; \sup_{x\in \mathcal{X}}\; \sup_{t\in (-\infty,\tau]} |
\widehat F_{T,0}^\beta ([t,\infty)\mid x) - F_{T,0}^\beta ([t,\infty)\mid x)|= o_{\mathbb{P}}(1).
$$
\end{lem}

\begin{proof}[Proof of Lemma \ref{F_C_et_al}]
Let us first investigate the uniform convergence of the estimated cumulative hazard measure $\widehat \Lambda_C(\cdot\mid x)$. For any $t\in(-\infty,\tau]$ let us write
\begin{multline*}
\widehat \Lambda_C((-\infty,t]\mid x) - \Lambda_C((-\infty,t]\mid x) = \int_{(-\infty,t]}\frac{\widehat H_0 (ds\mid x)}{\widehat H([s,\infty)\mid x)} - \int_{(-\infty,t]}\frac{H_0 (ds\mid x)}{H([s,\infty)\mid x)}\\
=\int_{(-\infty,t]}\left[ \frac{1}{\widehat H([s,\infty)\mid x)} -\frac{1}{ H([s,\infty)\mid x)} \right] \widehat H_0 (ds\mid x) + \int_{(-\infty,t]}\frac{\widehat H_0 (ds\mid x) - H_0 (ds\mid x)}{H([s,\infty)\mid x)}.
\end{multline*}
The integrals with respect to $\widehat H_0 (ds\mid x)$ are well defined, since, with probability tending to 1, for each $x\in\mathcal{X}$, the map $s\mapsto \widehat H_0 ((-\infty,s]\mid x),$ $s\in(-\infty,\tau],$ is a function of bounded variation. The uniform convergence Assumption (AC1) implies that
\begin{eqnarray*}
&& \hspace{-0.8cm}\sup_{x\in\mathcal{X}} \sup_{t\in (-\infty,\tau]} \left| \widehat \Lambda_C((-\infty,t]\mid x) - \Lambda_C((-\infty,t]\mid x) \right| \\
\!\!&\leq & \!\!\!c \; \sup_{x\in\mathcal{X}}\sup_{t\in (-\infty,\tau]} \frac{\left\{ \left| \widehat H_0([t,\infty)\!\mid \!x)\! -\! H_0([t,\infty)\!\mid \!x)\right|\!+\!\left| \widehat H([t,\infty)\mid x) \!-\! H([t,\infty)\!\mid \!x)\right|\right\}}{H([\tau,\infty)\mid \!x)^2},
\end{eqnarray*}
for some constant $c>0.$ Next, by Duhamel's identity (see Gill and Johansen 1990),
%\begin{multline*}
\begin{multline*}
\widehat F_C((t,\infty) \mid x) - F_C( (t,\infty) \mid x) = - F_C( (t,\infty) \mid x) \\
\times
 \int_{(-\infty,t]} \frac{\widehat F_C( [s,\infty) \mid x)}{F_C( (s,\infty) \mid x)} \left(\widehat \Lambda_C(ds\mid x) - \Lambda_C(ds\mid x) \right).
\end{multline*}
%\end{multline*}
Then, the uniform convergence of $\widehat F_C(\cdot \mid x)$ follows from the uniform convergence of  $\widehat \Lambda_C(\cdot \mid x)$ and condition (\ref{hhhh}). The same type of  arguments apply for $T_1(\beta,\widehat H_0,\widehat H_1)$, and hence we omit the details.

Next, since by conditions (\ref{hhhh}) and (AC4) we have
$$
\inf_{\beta \in B} \; \inf_{x\in \mathcal{X}} \; \inf_{t\in(-\infty,\tau]}\left[ H([t,\infty)\mid x) -  (1-\phi(x,\beta)) F_C ([t,\infty)\mid x)\right]>0,
$$
there exists some constant $c>0$ with the property that
$$
\mathbb{P}\left(\inf_{\beta\in B} \; \inf_{x\in \mathcal{X}} \; \inf_{t\in(-\infty,\tau]}\left[\widehat H([t,\infty)\mid x) - (1-\phi(x,\beta)) \widehat F_C ([t,\infty)\mid x)\right]\geq c>0\right) \rightarrow 1.
$$
Hence, the uniform convergence of $\widehat F_{T,0}^\beta(\cdot \mid x)$ follows.
\end{proof}

%\quad
%\smallskip

\begin{proof}[Proof of Theorem \ref{consist_prop}]
Let us write
%\begin{eqnarray*}
$$
F_{T,0}^\beta ( \{ Y_i \} \mid X_i) = \Lambda_{T,0}^\beta (\{ Y_i \} \mid X_i) F_{T,0}^\beta ([Y_i,\infty) \mid X_i) 
= H_1 (\{Y_i\} \mid X_i) \frac{F_{T,0}^\beta ([Y_i,\infty) \mid X_i)}{T_1(\beta, H_0,H_1)(Y_i,X_i)},
$$
%\end{eqnarray*}
where
$$
T_1(\beta, H_0,H_1)(t,x) = H([t,\infty)\mid x) - (1-\phi(x,\beta))F_C ([t,\infty)\mid x).
$$
Moreover, let 
$$
q_i(\beta, H_0,H_1) = q(\beta, H_0,H_1) (Y_i,\delta_i,X_i),
$$ 
where, for $ t \in \R,\; d\in\{0,1\},\; x\in\mathcal{X},$
\begin{multline*}
q(\beta, H_0,H_1)(t,d,x) = d \{\log \phi(x,\beta)+ \log F_{T,0}^\beta ([t,\infty) \mid x) - \log T_1(\beta, H_0,H_1)(t,x)  \}\\
+ (1-d )\log\{ \phi(x,\beta) F_{T,0}^\beta ([t,\infty)\mid x) + 1 - \phi(x,\beta) \}.
\end{multline*}
Let
$$
Q_n(\beta, H_0,H_1) = \frac{1}{n} \sum_{i=1}^n q_i(\beta, H_0, H_1).
$$
Similarly, let us consider $Q_n(\beta, \widehat H_0, \widehat H_1)$ that is defined as $Q_n(\beta, H_0,H_1),$ but with $H_0,$ $H_1,$ $F_C $ and $F_{T,0}^\beta$ replaced by
$\widehat H_0,$ $\widehat H_1,$  $\widehat F_C $ and $\widehat F_{T,0}^\beta$, respectively.
Then the estimator $\widehat \beta$ %defined 
in equation (\ref{estim_def}) becomes
%could be written as
$$
\widehat \beta  = \arg \max_{\beta \in B} Q_n(\beta, \widehat H_0, \widehat H_1).
$$
The first step is to check that
\begin{equation}\label{1st}
\sup_{\beta\in B} \left| Q_n(\beta, \widehat H_0, \widehat H_1) - Q_n(\beta, H_0, H_1) \right| = o_{\mathbb{P}}(1).
\end{equation}
This follows directly from Lemma \ref{F_C_et_al}.
Next, given our assumptions, it is easy to check that for any
$t\in(-\infty,\tau],$ $d\in\{0,1\},$ $ x\in\mathcal{X},$
$$
\left| q(\beta, H_0,H_1)(t,d,x) - q(\beta^\prime, H_0,H_1)(t,d,x)\right| \leq C \|\beta- \beta^\prime\|^a,\quad\forall \beta,\beta^\prime\in B,
$$
with $a>0$ from Assumption (AC3) and some constant $C$ depending only  on $c_1$ from Assumption (AC3) and the positive values $\inf_{x \in\mathcal{X} } H_1(\{\tau\}\mid x) $ and $\inf_{x \in\mathcal{X} }H_0((\tau,\infty)\mid x) .$  It follows that the class $\{ (t,d,x) \rightarrow q(\beta,H_0,H_1)(t,d,x) : \beta \in B \}$ is Glivenko-Cantelli. Hence, %by the uniform law of large numbers,
$$
\sup_{\beta\in B}\left| Q_n(\beta, H_0,H_1) - Q(\beta, H_0,H_1) \right| = o_{\mathbb{P}}(1),
$$
where $Q = \mathbb{E}(Q_n)$.  Finally, Proposition \ref{ident_lik} guarantees that
$$
\beta_0 = \arg\max_{\beta\in B} Q(\beta, H_0,H_1).
$$
Gathering the facts, we deduce that $\widehat \beta - \beta_0 =o_{\mathbb{P}}(1).$ 
\end{proof}

%\quad
%\smallskip

%\noindent
%{\bf Proof of Theorem \ref{cons}.}   \hfill $\Box$ \\[-.3cm]
%

%\quad
\begin{proof}[Proof of Theorem \ref{asno}.]
We show the asymptotic normality of our estimator by verifying the high-level conditions in Theorem 2 in Chen \emph{et al.}  (2003).   First of all, for the consistency we refer to Section \ref{consistency}, whereas conditions (2.1) and (2.2) in Chen \emph{et al.}  (2003) are satisfied by construction and thanks to assumption (AN1), respectively.   Concerning (2.3), first note that the expression inside the expected value in $\nabla_\eta M(\beta,\eta_0)[\eta-\eta_0]$ is linear in $\nabla_\eta T_j(\beta,\eta_0)[\eta-\eta_0]$ ($j=1,2,3,4$).  Hence, we will focus attention on the latter G\^ateaux derivatives.  First,
$$ \nabla_\eta T_1(\beta,\eta_0)[\eta-\eta_0](t,x) = (\eta_1 - \eta_{01} + \eta_2 - \eta_{02})(t,x) - (1-\phi(x,\beta)) \nabla_\eta T_2(\eta_0)[\eta-\eta_0](t,x). $$
Using Duhamel's formula (see Gill and Johansen 1990), we can write
\begin{eqnarray*}
\nabla_\eta T_2(\eta_0)[\eta-\eta_0](t,x) &\m = &\m - T_2(\eta_0)(t,x) \int_{-\infty<u<t} \frac{1}{(\eta_{01}+\eta_{02})(u,x) - \eta_{01}(\{u\},x)} \\
&\m &\m \times  \Big\{(\eta_1-\eta_{01})(du,x) - \frac{\eta_{01}(du,x) (\eta_1 + \eta_2 - \eta_{01} - \eta_{02})(u,x)}{(\eta_{01}+\eta_{02})(u,x)} \Big\}.
\end{eqnarray*}
In a similar way, we find that
\begin{eqnarray*}
 \nabla_\eta T_3(\beta,\eta_0)[\eta-\eta_0](t,x) 
& =& - T_3(\beta,\eta_0)(t,x) \int_{-\infty < u < t} \frac{1}{T_1(\beta,\eta_0)(u,x) - \eta_{02}(\{u\},x)} \\
&& \hspace*{.1cm} \times \Big\{(\eta_2-\eta_{02})(du,x) - \frac{\eta_{02}(du,x) \big[T_1(\beta,\eta)-T_1(\beta,\eta_0)\big](u,x)}{T_1(\beta,\eta_0)(u,x)} \Big\}.
\end{eqnarray*}
Finally,
\begin{eqnarray*}
&& \nabla_\eta T_4(\beta,\eta_0)[\eta-\eta_0](t,x) \\
&& = - \nabla_\beta \phi(x,\beta) \int_{(-\infty,t)} \left\{\frac{\nabla_\eta T_2(\eta_0)[\eta-\eta_0](s,x) \eta_{02}(ds,x) + T_2(\eta_0)(s,x) (\eta_2-\eta_{02})(ds,x)}{T_1(\beta,\eta_0)(s,x) \big[T_1(\beta,\eta_0)(s,x) - \eta_{02}(\{s\},x)\big]} \right. \\
&& \hspace*{2cm} - \frac{T_2(\eta_0)(s,x) \eta_{02}(ds,x) \nabla_\eta T_1(\beta,\eta_0)[\eta-\eta_0](s,x)}{\big[T_1(\beta,\eta_0)(s,x)\big]^2 \big[T_1(\beta,\eta_0)(s,x) - \eta_{02}(\{s\},x)\big]} \\
&& \hspace*{2cm} \left. - \frac{T_2(\eta_0)(s,x) \eta_{02}(ds,x) \big[\nabla_\eta T_1(\beta,\eta_0)[\eta-\eta_0](s,x) - (\eta_2-\eta_{02})(\{s\},x)\big]}{T_1(\beta,\eta_0)(s,x) \big[T_1(\beta,\eta_0)(s,x) - \eta_{02}(\{s\},x)\big]^2} \right\}.
\end{eqnarray*}
Note that all denominators in $\nabla_\eta T_j(\beta,\eta_0)[\eta-\eta_0](t,x)$ are bounded away from zero, thanks to (\ref{infinf}) and (\ref{infinfinf}).  By tedious but rather elementary arguments, it follows from these formulae that%, for some constant $C$,
$$
\| \nabla_\eta T_j(\beta,\eta_0)[\eta-\eta_0]- \nabla_\eta T_j(\beta_0,\eta_0)[\eta-\eta_0]\|_{\mathcal{H}} \leq C \|\beta-\beta_0\| \|\eta-\eta_0\|_{\mathcal{H}},
$$
for some constant $C$. Hence, it can be easily seen that $\nabla_\eta T_j(\beta,\eta_0)[\eta-\eta_0]$ satisfies the second property in assumption (2.3) in Chen \emph{et al.}  (2003), and hence the same holds true for $\nabla_\eta M(\beta,\eta_0)[\eta-\eta_0]$. Similarly, by decomposing $T_j(\beta,\eta) - T_j(\beta,\eta_0) - \nabla_\eta T_j(\beta,\eta_0)[\eta-\eta_0]$ using Taylor-type arguments (in $\eta$), the first property in assumption (2.3) is easily seen to hold true.

Next, conditions (2.4) and (2.6) are satisfied thanks to Assumption (AN4) and because it follows from the above calculations of $\nabla_\eta T_j(\beta,\eta_0)[\eta-\eta_0]$ ($j=1,2,3,4$) that
\begin{multline}\label{defphijk}
\nabla_\eta M(\beta_0,\eta_0)[\eta-\eta_0]  
 =  \sum_{k\in\{0,1\}}%^1
 \mathbb{E}\left[ \psi_{1k}(Y,X) \int_{-\infty<u<Y} \psi_{2k} (u,X)  d\left( (\eta_k-\eta_{0k})(u,X) \right) \right] \\%\nonumber \\
 + \sum_{k,\ell\in\{0,1\}}%^1
 \mathbb{E}\left[ \psi_{3k}(Y,X) \int_{-\infty<u<Y} \psi_{4k}(u,X) \psi_{5k} \Big((\eta_k-\eta_{0k}) (u,X) \Big)  dH_\ell(u\mid X) \right]%\nonumber
\end{multline}
for certain measurable functions $\psi_{jk}$ ($j=1,\ldots,5; k=0,1$).

It remains to verify condition (2.5).  Note that
\begin{eqnarray*}
 |m(t,\delta,x;\beta_2,\eta_2) -  m(t,\delta,x;\beta_1,\eta_1)| 
\le  C_1(t,\delta,x) \|\beta_2 - \beta_1\| + C_2(t,\delta,x) \|\eta_2 - \eta_1\|_{\mathcal H}
\end{eqnarray*}
for some functions $C_j$ satisfying ${\mathbb E}[C_j^2(Y,\delta,X)] < \infty$ ($j=1,2$), and hence (2.5) follows from assumption (AN5) and Theorem 3 in Chen \emph{et al.} (2003). This finishes the proof.
\end{proof}
%\hfill $\Box$ \\[-.3cm]

%\quad
%\smallskip

\begin{proof}[Proof of Theorem \ref{boot}.]
%\noindent
%{\bf Proof of Theorem \ref{boot}.}
To prove this theorem we will check the conditions of Theorem B in Chen \emph{et al.} (2003), which gives high level conditions under which the naive bootstrap is consistent.  The only difference between their setting and our setting is that we are proving bootstrap consistency in $\mathbb P$-probability, whereas their result holds true a.s.\ $[\mathbb P]$.  As a consequence, in their high level conditions we can replace all a.s.\ $[\mathbb P]$ statements by the corresponding statements in $\mathbb P$-probability.

First of all, it follows from assumption (AN1) that condition (2.2) in Chen \emph{et al.} (2003) holds with $\eta_0$ replaced by any $\eta$ in a neighborhood of $\eta_0$, and from the proof of Theorem \ref{asno} it follows that the same holds true for condition (2.3).   Next, conditions (2.4B) and (2.6B) in Chen \emph{et al.} %(2003) 
follow from the fact that we assume that assumption (AN4) continues to hold true if we replace $\widehat H_k-H_k$ by $\widehat H_k^*-\widehat H_k$ ($k=0,1$).  It remains to verify condition (2.5'B) in Chen \emph{et al.} %(2003).  
This follows from  Theorem 3 in Chen \emph{et al.}, %(2003), 
whose conditions have been verified already for % in the proof of 
our Theorem \ref{asno}.
\end{proof}
%\hfill $\Box$ \\[-.3cm]

%\newpage

%\section*{Acknowledgements}
%Valentin Patilea gratefully acknowledges XXXX


\begin{thebibliography}{9}

\bibitem{avk}
\textsc{Akritas, M.G. \& Van Keilegom, I.} (2001).
Nonparametric estimation of the residual distribution.
\textsl{Scand. J. Statist.} \textbf{28}, 549--568.

\bibitem{Amico}
\textsc{Amico, M., Legrand, C. \& Van Keilegom, I.} (2017).
The single-index/Cox mixture cure model
(submitted). 

%\bibitem{AVK}
%\textsc{Amico, M. \& Van Keilegom, I.} (2017).
%A review on cure models
%(in preparation). 

\bibitem{Boag}
\textsc{Boag, J.W.} (1949). 
Maximum likelihood estimates of the proportion of patients cured by cancer therapy. 
\textsl{J. Roy. Statist. Soc. - Series B} \textbf{11}, 15--53.

\bibitem{r01}
\textsc{Chen, X., Linton, O. \& Van Keilegom, I.} (2003).
Estimation of semiparametric models when the criterion function is not smooth. \textsl{Econometrica} \textbf{71}, 1591--1608.

%\bibitem{r1}
%\textsc{Cox, D.R.} (1972).
%Regression models and life tables (with discussion).
%\textsl{J. Roy. Statist. Soc. - Ser. B} \textbf{34}, 187--220.

\bibitem{r2}
\textsc{Fang, H.B., Li, G. \& Sun,  J.} (2005).
Maximum likelihood estimation in a semiparametric logistic/proportional-hazards mixture model. \textsl{Scand. J.  Statist.} \textbf{32}, 59--75.

\bibitem{r21}
\textsc{Farewell, V.T.}  (1982).
The use of mixture models for the analysis of survival data with long-term survivors. \emph{Biometrics} \textbf{38}, 1041--1046.

\bibitem{r21g}
\textsc{Gill, R.D.} (1994).
\emph{Lectures on survival analysis.} Lectures on probability theory:
Ecole d'\'et\'e de probabilit\'es de Saint-Flour XXII. Lecture notes in mathematics 1581.
Springer.

\bibitem{r21gj}
\textsc{Gill, R.D. \& Johansen, S.} (1990).
A survey of product-integration with a view toward application in survival analysis.
\textsl{Ann. Statist.} \textbf{18}, 1501--1555.

\bibitem{r4}
\textsc{Kuk, A.Y.C. \& Chen, C.-H.} (1992).
A mixture model combining logistic regression with proportional hazards regression.  \textsl{Biometrika} \textbf{79}, 531--541.

\bibitem{llr}
\textsc{Li, Q., Lin, J. \& Racine, J.S.} (2013).
Optimal bandwidth selection for nonparametric conditional distribution and quantile functions.
\textsl{J. Buss. Econ. Statist.} \textbf{31}, 57--65.

%\bibitem{r4a}
%\textsc{Li, C.S., Taylor, J.M.G. \& Sy, J.P.} (2001).
%Identifiability of cure models.
%\emph{Statist. \& Probab. Lett.} \textbf{54}, 389--395.

\bibitem{Lopez}
\textsc{L\'opez-Cheda, A., Cao, R., J\'acome M.A. \& Van Keilegom, I.} (2017).
Nonparametric incidence estimation and bootstrap bandwidth selection in mixture cure models.
\textsl{Comput. Statist. Data Anal.} \textbf{105}, 144--165.

\bibitem{r4a2}
\textsc{Lopez, O.} (2011).
Nonparametric estimation of the multivariate distribution function in a censored regression
model with applications.
\emph{Commun. Stat. - Theory Meth.} \textbf{40}, 2639--2660.

\bibitem{r4b}
\textsc{Lu, W.} (2008).
Maximum likelihood estimation in the proportional hazards cure model.
\textsl{Ann. Inst. Stat. Math.} \textbf{60}, 545--574.

\bibitem{r41}
\textsc{Maller, R.A. \& Zhou, S.} (1996).
\emph{Survival Analysis with Long Term Survivors}.
Wiley, New York.

\bibitem{mee}
\textsc{Meeker, W.Q.} (1987). 
Limited failure population life tests: Application to integrated circuit reliability.
\emph{Technometrics} \textbf{29}, 51--65. 

\bibitem{oth}
\textsc{Othus, M., Li, Y.  \&    Tiwari, R.C.} (2009). 
A class of semiparametric mixture cure survival models with dependent censoring. 
\emph{J. Amer. Statist. Assoc.} \textbf{104}, 1241--1250.

\bibitem{Peng}
\textsc{Peng, Y. \& Taylor, J.M.G.} (2014). 
Cure models. 
In: Klein, J., van Houwelingen, H., Ibrahim, J. G., and Scheike, T. H., editors, {\it Handbook of Survival Analysis}, Handbooks of Modern Statistical Methods series, chapter 6, pages 113-134. Chapman \& Hall, Boca Raton, FL, USA.

\bibitem{r41a}
\textsc{Schmidt, P. \& Witte, A.D.} (1989).
Predicting criminal recidivism using split population survival time models.
\emph{J. Econometrics} \textbf{40}, 141--159.

\bibitem{Sy}
\textsc{Sy, J.P. \& Taylor, J.M.G.} (2000). 
Estimation in a Cox proportional hazards cure model. 
\emph{Biometrics} \textbf{56}, 227--236.

\bibitem{r44}
\textsc{Taylor, J.M.G.} (1995).
Semi-parametric estimation in failure time mixture models.
\emph{Biometrics} \textbf{51}, 899--907.

%\bibitem{r44a}
%\textsc{Tsodikov, A.D., Ibrahim, J.G. \& Yakovlev, A.Y.} (2003).
%Estimating  cure rates from survival data: an alternative to two-component mixture models.
%\emph{J. Amer. Statist. Assoc.} \textbf{98}, 1063--1078.
	
\bibitem{r26_a}
\textsc{van der Vaart, A.D.} (1998).
\textsl{Asymptotic Statistics.}
Cambridge University Press.

%\bibitem{r26}
%\textsc{van der Vaart, A.D. \& Wellner, J.A.} (1996).
%\textsl{Weak Convergence and Empirical Processes.}
%Springer, New York.

\bibitem{Wang}
\textsc{Wang, Y., Klijn, J.G.M., Sieuwerts, A.M., Look, M.P., Yang, F., Talantov, D., Timmermans, M., Meijet-van Gelder, M.E.M., Yu, J., Jatkoe, T., Berns, E.M.J.J., Atkins, D. \& Foekens, J.A.} (2005). 
Gene-expression profiles to predict distant metastasis of lymph-node-negative primary breast cancer.
\emph{The Lancet} \textbf{365}, 671--679.

\bibitem{Xu}
\textsc{Xu, J. \& Peng, Y.} (2014).
Nonparametric cure rate estimation with covariates.
\emph{Canad. J. Statist.} \textbf{42}, 1--17.

%\bibitem{r27}
%\textsc{Yin, G. \& Ibrahim, J. G.} (2005).
%Cure rate models: a unified approach.
%\emph{Canad. J. Statist.} \textbf{33}, 559--570.

%\bibitem{r28}
%\textsc{Zheng, D., Yin, G. \& Ibrahim, J.G.} (2006).
%Semiparametric transformation models for survival data with a cure fraction.
%\emph{J. Amer. Statist. Assoc.} \textbf{101}, 670--684.


\end{thebibliography}
\end{document}